\documentclass[12pt]{article}
\title{Circular Planar Resistor Networks with Nonlinear and Signed Conductors}
\author{Will Johnson}
\usepackage{amsfonts}
\usepackage{amsmath}
\usepackage{amsmath, amssymb, amsthm}    	
\usepackage{fullpage} 	
\usepackage{graphicx}	
\usepackage{float}
\usepackage{verbatim}
\usepackage{hyperref}
\usepackage{cite}
\usepackage{color}
\usepackage{subfig}

\newcommand{\dist}{\operatorname{d}}
\newcommand{\rank}{\operatorname{rank}}

\newtheorem{theorem}{Theorem}[section] 
\newtheorem{lemma}[theorem]{Lemma} 
\newtheorem{proposition}[theorem]{Proposition}
\newtheorem{corollary}[theorem]{Corollary}
\newtheorem{definition}[theorem]{Definition}

\newtheorem{conjecture}[theorem]{Conjecture}
\newtheorem{remark}[theorem]{Remark}

\begin{document}
\maketitle

\abstract{We consider the inverse boundary value problem in the case of discrete electrical networks containing nonlinear (non-ohmic) resistors.
Generalizing work of Curtis, Ingerman, Morrow, Colin de Verdi\`ere, Gitler, and Vertigan, we characterize the circular planar graphs for which the inverse boundary value problem has a solution in this generalized non-linear setting.  The answer is the same as in the linear setting.  Our method of proof never requires that the resistors
behave in a continuous or monotone fashion; this allows us to recover signed conductances in many cases.
We apply this to the problem of recovery in graphs that are not circular planar.  We also use our results to make a frivolous knot-theoretic statement,
and to slightly generalize a fact proved by Lam and Pylyavskyy about factorization schemes in their electrical linear group. }

\section{Introduction}\label{sec:intro}
Consider a network of resistors, containing several boundary nodes where current is allowed to flow in and out.  If we fix the voltages at these boundary nodes,
all interior voltages and currents will be determined, and in particular the amount of current flowing in and out at each boundary node is determined.  Thus,
boundary voltages determine boundary currents.  The relationship between boundary voltages
and boundary currents can be encapsulated in a matrix $\Lambda$, the \emph{response matrix} or \emph{Dirichlet-to-Neumann matrix}.

Viewing the circuit as a black box whose internals are hidden, the matrix $\Lambda$ exactly describes the behavior of the circuit.  It is never possible
to determine the internal structure of a circuit from $\Lambda$, but if we are given the structure of the underlying graph,
then in many cases we can determine all the resistances from $\Lambda$.  The problem of recovering the resistances from boundary measurements
is a discrete version of the inverse boundary value problem first posed by Calderon.

We say that a graph
is \emph{recoverable} if the recovery problem has a solution.  Working independently, \cite{CIM} and \cite{ReseauxElectriques} completely analyzed the case of
circular planar networks, finding a combinatorial criterion for recoverability, and finding algorithms to carry out the recovery when possible.
In this paper, we generalize these results to allow for non-linear (non-ohmic) conductors.

In \S \ref{sec:forward-primal}-\ref{sec:forward-dual}, we consider the forward problems, for fixed networks.  Assuming that current at each edge is given
as a continuous, monotone, zero-preserving function of voltage at each edge in the graph, we show that boundary voltages determine all internal currents
uniquely.  In particular, there is a well-defined Dirichlet-to-Neumann map from boundary voltages to boundary currents.
Dually, if voltage is given as a continuous monotone zero-preserving
function of current at each edge, then boundary currents more or less determine all interior currents, and there is essentially a well-defined Neumann-to-Dirichlet
map from boundary voltages to boundary currents.

In subsequent sections we consider the inverse problem of determining the conductance functions at each edge from the Dirichlet-to-Neumann map
or the Neumann-to-Dirichlet map.  We are forced to assume that the conductance functions are bijective and zero-preserving, but our approach somehow doesn't require
continuity or monotonicity, rather surprisingly.  Our method of proof is apparently a variant of the approach used in \cite{ReseauxElectriques}, though
substantial modifications are necessary, because Y-$\Delta$ transformations are no longer available.
Section \ref{sec:convexity} is devoted to showing that a certain notion of convexity behaves well in simple pseudoline arrangements.  This notion of convexity is
used to model the propagation of information in unusual boundary-value problems which we use to read off conductances near the boundary, in a layer-stripping approach.

Our main result is the following fact: if $G$ is a circular planar graph for which recovery is possible in
the usual setting with linear conductors and positive conductances, then recovery is still possible when (sufficiently well-behaved) nonlinear and signed conductors
are allowed.  In \S \ref{sec:applications}, we apply this to two settings in which negative conductances naturally appear.  In \S \ref{sec:el2n}, we restate our results
from the point of view of Lam and Pylyavskyy's electrical linear group and consider the analogy between the electrical linear group and total positivity in
the general linear group.

\tableofcontents

\section{The Dirichlet-to-Neumann Map}\label{sec:forward-primal}
In this section and the next, we prove some simple facts about general networks with nonlinear conductors.\footnote{We include these elementary results primarily because
they were previously unknown within the Mathematics REU at the University of Washington (which focuses on the recovery problem in electrical networks).}
\begin{definition}A \emph{graph with boundary} $\Gamma$ is a graph whose set of vertices $V$ have been partitioned into a set of boundary vertices $\partial \Gamma$
and interior vertices $\Gamma^{\textrm{int}}$.
\end{definition}
We allow our graphs to contain multiple edges and self-loops, though the latter are completely irrelevant in what follows.  If $e$ is a directed
edge in our graph, we let $\overline{e}$ be the edge going in the opposite direction.
\begin{definition}
If $\Gamma$ is a graph with boundary, a \emph{conductance network on $\Gamma$} is a collection of functions $\gamma_e : \mathbb{R} \to \mathbb{R}$, one for
each directed edge $e$ in $\Gamma$, such that $\gamma_e(-x) = -\gamma_{\overline{e}}(x)$
and such that each $\gamma_e$ is zero-preserving ($\gamma_e(0) = 0$), monotone ($\gamma_e(x) \ge \gamma_e(y)$ if $x \ge y$), and continuous.
\end{definition}
We abuse notation and use $\Gamma$ for both the network and the graph.
\begin{definition}
A \emph{voltage function} on a graph with boundary $\Gamma$ is a function $\phi$ from the vertices of $\Gamma$ to $\mathbb{R}$.  A \emph{current function}
on $\Gamma$ is a function $\iota$ from the directed edges of $\Gamma$ to $\mathbb{R}$, such that
\[ \iota(e) = -\iota(\overline{e})\]
and such that for every vertex $v$ in the \emph{interior} of $\Gamma$, $\sum_{e} \iota(e) = 0$, where the sum ranges over all edges leaving the vertex $v$.  A
voltage function $\phi$ and current function $\iota$ are \emph{compatible (with respect to some conductance network on $\Gamma$)} if
\[ \iota(e) = \gamma_e(\phi(x) - \phi(y))\]
whenever $e$ is an edge leading from vertex $x$ to vertex $y$.  Given such a pair, the \emph{boundary voltage function} is the restriction
of $\phi$ to $\partial \Gamma$, while the \emph{boundary current function} is the function that assigns to each vertex $v \in \partial \Gamma$,
the sum $\sum_e \iota(e)$ where $e$ ranges over the edges leaving from $v$.
\end{definition}
In other words, the interpretation of the function $\gamma_e$ is the function that specifies current in terms of voltage along that edge.  In the
\emph{linear} case, $\gamma_e(x) = c_e x$, where $c_e$ is the conductance along edge $e$.

Given a conductance network and a set of boundary voltages, the \emph{Dirichlet boundary value problem} is to find a voltage function extending the given boundary
values, together with a compatible
current function.  In general, the solution to the Dirichlet problem may not be unique, but solutions always exist, and
the current function is uniquely determined:
\begin{theorem}\label{dirichlet-problem}
Let $\Gamma$ be a conductance network and let $f$ be any function from $\partial \Gamma$ to $\mathbb{R}$.  Then there is a voltage function
$\phi$ and a compatible current function $\iota$ on $\Gamma$ that has $f$ as the boundary voltage function.  Moreover, $\iota$ is uniquely determined.
\end{theorem}
Note that $\phi$ need not be uniquely determined.  This occurs for example if every function $\gamma_e$ is the zero function and there are any interior vertices.
\begin{proof}
We first show existence.  Let $\mathbb{R}^V$ be the vector space of all voltage functions.  Let $W$ be the affine subspace of $\mathbb{R}^V$
containing only functions which agree with $f$ on $\partial \Gamma$.  For each directed edge $e$, let
\[ q_e(x) = \int_0^x \gamma_e(t) dt\]
This function is $C^1$, nonnegative, and convex because of the properties of $\gamma_e$.  We also have $q_e(0) = 0$, and $q_e(-x) = q_{\overline{e}}(x)$.  Moreover,
if $x$ and $y$ are both nonnegative (or both nonpositive), and $|x| \le |y|$, then
\begin{equation} q_e(x) \le q_e(y).\label{bowl-shape}\end{equation}
This is easy to verify.

For
$\phi \in \mathbb{R}^V$, we define the \emph{pseudopower} $Q(\phi)$ to be
\[ \sum_{e \in E} q_e(\phi(\textrm{start}(e)) - \phi(\textrm{end}(e)))\]
where $\textrm{start}(e)$ and $\textrm{end}(e)$ are the start and end of $e$, and the sum is over all directed edges.  Note that the terms
for $e$ and $\overline{e}$ agree, so we are essentially counting each edge twice.

In the linear case, this function is the usual power, the sum over voltage times current for each edge.  In the more general case, however,
it is not this value, so we call it the pseudopower.

As a sum of nonnegative convex $C^1$ functions, pseudopower is a nonnegative
convex $C^1$ function on $\mathbb{R}^V$, and also on $W$.  We aim to show that $Q(\cdot)$ attains a minimum
on $W$.  To see this, first let $M$ be the maximum of $|f(v)|$ as $v$ ranges over $\partial \Gamma$.  Let $K$ be the hypercube subset of $W$
containing only voltage functions which take values in $[-M,M]$.  Then $Q(\cdot)$ attains a minimum on $K$ because $K$ is compact.  Let $\phi_0$
be such a minimum on $K$.  We claim that $\phi_0$ is in fact a global minimum.  To see this, let $\phi$ be any other voltage function in $W$.
Let $\mu:\mathbb{R} \to [-M,M]$ be the retraction of $\mathbb{R}$ onto $[-M,M]$ that sends $(M,\infty)$ to $M$ and $(-\infty,-M)$ to $-M$.  Let
$\phi' = \mu \circ \phi$.  Then clearly $\phi'$ still agrees with $f$ on the boundary, so $\phi' \in W$, and in fact $\phi' \in K$.  Moreover,
because $\mu$ is a monotone short map and because the $q_e$ functions satisfy Equation (\ref{bowl-shape}), it follows easily that
\[ Q(\phi) \ge Q(\phi') \ge Q(\phi_0).\]
Since $\phi \in W$ was arbitrary, $\phi_0$ does indeed minimize $Q$ on all of $W$.

For any directed edge $e$ from vertex $x$ to vertex $y$, let $\iota(e) = \gamma_e(\phi(x) - \phi(y))$.  Then $\iota$ and $\phi$ will be compatible
current and voltage functions as long as $\iota$ is a current function.  The property that $\iota(e) = -\iota(\overline{e})$ follows easily
from properties of $\gamma_e$, so it remains to show that for every interior vertex $v$,
\[ \sum_{e \textrm{ with }v = \textrm{start}(e)} \iota(e) = 0.\]
This follows directly by taking the partial derivative of $Q(\phi)$ with respect to $\phi(v)$, and using the fact that $\phi_0$ is a critical point.
Specifically, the partial derivative
\[ \frac{\partial Q(\phi)}{\phi(v)} = \sum_{e \textrm{ with }v = \textrm{start}(e)} q_e'(\phi(v) - \phi(\textrm{end}(e)))
 - \sum_{e \textrm{ with }v = \textrm{end}(e)} q_e'(\phi(\textrm{start}(e)) - \phi(v)) \]
\[ = \sum_{e \textrm{ with }v = \textrm{start}(e)} \gamma_e(\phi(v) - \phi(\textrm{end}(e)))
- \sum_{e \textrm{ with }v = \textrm{end}(e)} \gamma_e(\phi(\textrm{start}(e)) - \phi(v)) =\]
\[ 2
\sum_{e \textrm{ with }v = \textrm{start}(e)} \gamma_e(\phi(v) - \phi(\textrm{end}(e)))\]
where the last line follows by the change of variables $e \leftrightarrow \overline{e}$ in the second sum, and
the fact that
\[ -\gamma_e(\phi(\textrm{start}(e)) - \phi(\textrm{end}(e))) = \gamma_{\overline{e}}(\phi(\textrm{start}(\overline{e})) - \phi(\textrm{end}(\overline{e}))).\]

This establishes existence.  It remains to show that the current function $\iota$ is unique.  Suppose we had two current functions $\iota_1$ and $\iota_2$
respectively compatible with two voltage functions $\phi_1$ and $\phi_2$, both of which agreed with $f$ on $\partial \Gamma$.  Define a new conductance network
$\hat{\Gamma}$ on the same graph by letting
\[ \hat{\gamma_e}(x) = \gamma_e(x + \phi_2(\textrm{start}(e)) - \phi_2(\textrm{end})(e)) - \iota_2(e)\]
This is a valid conductance function because
\[ \gamma_e(\phi_2(\textrm{start}(e)) - \phi_2(\textrm{end})(e)) = \iota_2(e)\]
implies that $\hat{\gamma_e}(0) = 0$, and the other properties are easy to verify.  One easily verifies that $\hat{\phi} = \phi_1 - \phi_2$ and
$\hat{\iota} = \iota_1 - \iota_2$ are now compatible voltage and current functions for $\hat{\Gamma}$.  Moreover, the boundary values
of $\hat{\phi}$ are all zero, because $\phi_1$ and $\phi_2$ agreed on the boundary.  So we are reduced to showing that if the boundary voltages vanish,
all currents must vanish.

Consider the directed subgraph of $\hat{\Gamma}$ in which we include a directed edge $e$ if and only if $\hat{\iota}(e) > 0$.  Since $\hat{\iota}$ equals
\[ \hat{\gamma_e}(\hat{\phi}(\textrm{start}(e)) - \hat{\phi}(\textrm{end}(e)))\]
and $\hat{\gamma_e}$ is monotone and zero-preserving, every edge $e$ in the subgraph must also satisfy
\[ \hat{\phi}(\textrm{start}(e)) > \hat{\phi}(\textrm{end}(e))\]
In particular then, our subgraph must be acyclic.  We claim that it has no interior `sources' or `sinks' in the following sense: there cannot be any
interior node $v$ which has at least one coming in and none coming out, or at least one edge going out and none coming in.  These follow from the facts
that the sum of $\hat{\iota}(e)$ for the edges incident to $v$ is zero, and a sum of positive and nonnegative numbers cannot be zero.

For $x, y \in V$, let $x \to y$ indicate that there is a chain of edges in the subgraph leading from $x$ to $y$, oriented in the correct way.  This relation
is obviously transitive, and irreflexive because $x \to y$ implies that $\hat{\phi(x)} > \hat{\phi(y)}$.  Suppose for the sake of contradiction that there is at least one edge $e$ in the subgraph.  Let $e$ go from $x_0$ to $y_0$.  Then choose $y$ with $y = y_0$ or $y \to y_0$, that is maximal in the sense that no $y'$ exists
with $y \to y'$.  Similarly choose a minimal $x'$ with $x' \to x_0$ or $x' = x_0$.  Because there are no internal sources or sinks, $x_0$ and $y_0$ must be boundary
nodes.  We also have $x_0 \to y_0$.  Thus $\hat{\phi}(x_0) > \hat{\phi}(y_0)$, contradicting the fact that $\hat{\phi}$ vanishes on the boundary.

So this subgraph is empty.  This means that for every edge $e$, $\hat{\iota}(e) \le 0$.  But then for every edge $e$ we also have $\hat{\iota}(\overline{e})
= -\hat{\iota}(e) \le 0$.  So $\hat{\iota}$ vanishes everywhere, and $\iota_1$ and $\iota_2$ agree.
\end{proof}
Since the current function determines the boundary current function, it follows that there is a well-defined map, the \emph{Dirichlet-to-Neumann map},
from boundary voltage functions to boundary current functions.

\section{The Neumann-to-Dirichlet Map}\label{sec:forward-dual}
This section is completely dual to the previous one.  Here we assume that voltage is given as a suitably nice function of current, and show essentially that
boundary currents uniquely determine boundary voltages.  There are some caveats, however: a given set of boundary currents must add up to zero on each connected
component of a graph, or else there are no solutions at all.  Additionally, there is latitude in choosing the voltages: adding a constant to each
value of a voltage function preserves compatibility.  In fact, we can add a separate constant on each connected component of $\Gamma$.

\begin{definition}
If $\Gamma$ is a graph with boundary, a \emph{resistance network on $\Gamma$} is a function that assigns
to each directed edge $e$ a function $\rho_e:\mathbb{R} \to \mathbb{R}$ which is
\begin{itemize}
\item zero-preserving: $\rho_e(0) = 0$
\item monotone: $\rho_e(x) \ge \rho_e(y)$ if $x \ge y$
\item continuous
\end{itemize}
such that $\rho_e(-x) = -\rho_{\overline{e}}(x)$.
\end{definition}
\begin{definition}
A voltage function $\phi$ and current function $\iota$ on $\Gamma$ are \emph{compatible (with respect to some resistance network on $\Gamma$)} if
\[ \rho_e(\iota(e)) = \phi(x) - \phi(y)\]
whenever $e$ is an edge leading from vertex $x$ to vertex $y$.  Given such a pair, the \emph{boundary voltage function} is the restriction
of $\phi$ to $\partial \Gamma$, while the \emph{boundary current function} is the function that assigns to each vertex $v \in \partial \Gamma$,
the sum $\sum_e \iota(e)$ where $e$ ranges over the edges leaving from $v$.
\end{definition}
In other words, the interpretation of the function $\rho_e$ is the function that specifies voltage in terms of current along that edge.  In the
\emph{linear} case, $\rho_e(x) = r_ex$, where $r_e$ is the resistance (reciprocal of conductance) along edge $e$.

The dual to Theorem~\ref{dirichlet-problem} is the following:
\begin{theorem}\label{neumann-problem}
Let $f$ be a real-valued function on $\partial \Gamma$, summing to zero on each connected component of $\Gamma$.  Then there is a voltage function
$\phi$ and a compatible current function $\iota$ with $f$ as the boundary current function.  Moreover, given two such solutions
$\phi_1$ and $\phi_2$, $\phi_1 - \phi_2$ is constant on each connected component of $\Gamma$.
\end{theorem}
\begin{proof}
We first show existence.  Let $W$ be the affine space of current functions on $\Gamma$ which have $f$ as their boundary currents.  This space is
nonempty because of the requirement that $f$ sum to zero on each connected component - we leave the proof of this as an exercise to the reader.
For each edge $e$, let
\[ q_e(x) = \int_0^x \rho_e(t) dt\]
As before, these functions are nonnegative, convex, and $C^1$ and satisfy Equation (\ref{bowl-shape}).  For $\iota \in W$, let $Q(\iota)$ be the sum
\[ Q(\iota) = \sum_e q_e(\iota(e))\]
Then as before, $Q(\cdot)$ is convex, nonnegative, and $C^1$.

Let $M = \sum_{v \in \partial \Gamma} |f(v)|$ and let $K$ be the set of all current functions $\iota \in W$ for which there are no directed cycles of edges
$e_1, e_2, \ldots, e_n$ with $\iota(e_i) > 0$ for all $i$.  We claim that $K$ is compact, because $|\iota(e)| \le M$ for $\iota \in K$
and $e$ any edge.  This is true because the acyclicity of $\iota$ implies that we can partition the vertices of $\Gamma$ into a set of
vertices `upstream' from $e$ and a set of vertices `downstream' from $e$, and the total current flowing through $e$ must be
at most the total current flowing into the upstream set, which is at most $M$.  So $K$ is bounded, and clearly closed.  Thus $Q$ attains a
minimum on $K$, at some $\iota_0$.

As before, we wish to show that $\iota_0$ is a global minimum.  To see this, take any $\iota \in W \setminus K$.  Then $\iota$ contains a directed
cycle of edges along which current is flowing in a positive direction.  Letting $m$ be the minimum of these positive current values, create
another current function $\iota'$ by subtracting $m$ from the current along each of these edges (and adding $m$ to the currents along the edges
going in the opposite directions).  Then clearly $\iota'$ is still a current function, and it has the same boundary currents as $\iota$.  Moreover,
since the current did not change in sign along any edge, and did not increase in magnitude along any edge, it follows by the analogue
of Equation (\ref{bowl-shape}) that $Q(\iota') \le Q(\iota)$.  Now $\iota'$ may have more cycles, and so may not yet be in $K$.  However,
by repeating this procedure at most finitely many times, we get an $\iota'' \in K$ with $Q(\iota) \ge Q(\iota'') \ge Q(\iota_0)$.  The process
must terminate after finitely many steps because the number of edges along which no current is flowing increases at each step.  Since $\iota$ was
arbitrary, $\iota_0$ is indeed a global minimum.

Take $\iota = \iota_0$.  Then we need to produce a voltage function $\phi$ for which
\[ \phi(x) - \phi(y) = \rho_e(\iota(e))\]
whenever $e$ is an edge from $x$ to $y$.  This is possible exactly when
\begin{equation} \sum_i \rho_{e_i}(\iota(e_i)) = 0\label{cycle-condition}\end{equation}
for every directed cycle $e_1, e_2, \ldots, e_n$.  To see that Equation (\ref{cycle-condition}) holds, take our cycle $e_1, e_2, \ldots, e_n$,
and let $\psi$ be the current function on $\Gamma$ which takes the value $1$ along each of the $e_i$, $-1$ along each $\overline{e_i}$, and 0 elsewhere.
Then $\psi$ is a current function with boundary current zero, so $\iota + t\psi \in W$ for every $t \in \mathbb{R}$.  Then the fact that $\iota$ is
a global minimum of $Q(\cdot)$ on $W$ implies that
\[ \frac{\partial Q(\iota + t\psi)}{\partial t}\]
vanishes at $t = 0$.
But by an easy direct calculation, this gives Equation (\ref{cycle-condition}).

Next we show that $\phi$ is essentially uniquely determined.  Let $\phi_1$ and $\phi_2$ be two solutions, respectively compatible
with $\iota_1$ and $\iota_2$.  As in the previous section, we can shift the resistance functions and produce a new network
which has $\phi_1 - \phi_2$ as a voltage function compatible with the current function $\iota_1 - \iota_2$.  So without loss of generality,
$f$ is identically zero, and we need to show that if $\phi$ and $\iota$ are compatible voltage and current functions with boundary current zero,
then $\phi$ is constant on each connected component.

Similar to the previous theorem, we take a subgraph containing all edges $e$ for which
\[ \phi(\textrm{start}(e)) > \phi(\textrm{end}(e))\]
Since voltage is given as a monotone function of current, this implies that $\iota(e) > 0$ too.  As before, there can then be no cycles and no
internal sources and sinks.  But in fact, since the boundary current is also zero, there can be no sources or sinks on the boundary either.  But
any finite directed acyclic graph without sources or sinks must be empty.  Therefore, for every edge $e$ we must have
\[ \phi(\textrm{start}(e)) \le \phi(\textrm{end}(e))\]
But then applying this to $\overline{e}$, we see that
\[ \phi(\textrm{start}(e)) \ge \phi(\textrm{start}(e))\]
too.  Thus $\phi$ must be constant on each connected component.
\end{proof}

Let $A$ be the vector space of real-valued functions $f$ on $\partial \Gamma$ with the property that for every connected component $C$ of $\Gamma$,
\[ \sum_{v \in C \cap \partial \Gamma} f(v) = 0.\]
Then it is easy to see that there is a well-defined Neumann-to-Dirichlet map $A \to A$ from boundary voltage functions in $A$ to boundary current functions in
$A$.  
In the case of a conductance or resistance network where the conductance or resistance functions are \emph{bijections},
it follows by combining Theorems~\ref{dirichlet-problem} and \ref{neumann-problem} that we have a well-defined bijection between boundary voltage
functions in $A$ and boundary current functions in $A$.
In other words,
the Dirichlet-to-Neumann map and the Neumann-to-Dirichlet map are essentially inverses of each other.  Also, boundary voltages uniquely determine interior
voltages in this case, so the solution to the Dirichlet problem is unique.

\section{The Inverse Problem}\label{sec:inverse-problem}
Given the Dirichlet-to-Neumann map or the Neumann-to-Dirichlet map, can we determine the resistance or conductance functions?  Without putting
additional restrictions on our conductance or resistance functions, the answer is no in almost all cases, because of the following problem.
Consider the Y-shaped network of Figure~\ref{fig1-y-3}.
If we put conductance functions of $\arctan x$ on two of the three edges, then the total current through the third edge can never exceed $\pi$.
Consequently, if we put a conductance function of $\gamma_3(x) = x$ on the third edge, this yields the same Dirichlet-to-Neumann map as
the alternative conductance function
\[ \tilde{\gamma}_3(x) =
\begin{cases}
x & \textit{if } - \pi \le x \le \pi \\
\pi & \textit{if } x \ge \pi \\
-\pi & \textit{if } x \le -\pi
\end{cases}
\]
Similar arguments work as long as any interior vertices exist.  In the dual case where voltage is given as a function of current, a similar problem
occurs in any graph that contains cycles.

\begin{figure}
\centering
\def \svgwidth{2in}
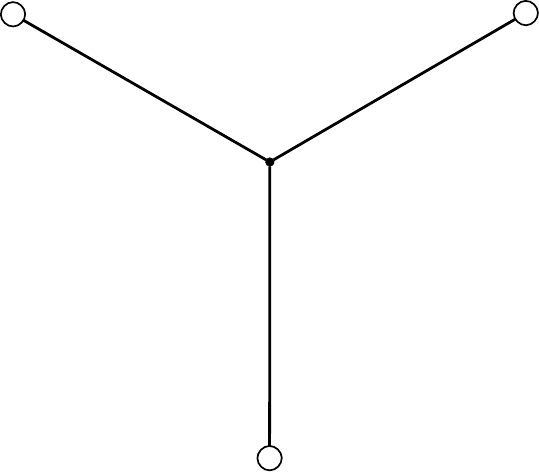
\caption{A graph for which the general recovery problem is impossible.  The dark middle vertex is an interior vertex, while the light vertices
are boundary vertices.}
\label{fig1-y-3}
\end{figure}

To resolve this difficulty, we instead consider networks where the conductance functions or resistance functions are bijections:
\begin{definition}
A \emph{bijective network} is a graph $\Gamma$ with boundary, and a function that assigns to each directed edge $e$ in $\Gamma$,
a \emph{conductance function} $\gamma_e:\mathbb{R} \to \mathbb{R}$, which satisfies the following properties:
\begin{itemize}
\item $\gamma_e(0) = 0$
\item $\gamma_e$ is a bijection for every $e$
\item $\gamma_e(-x) = -\gamma_{\overline{e}}(x)$
\end{itemize}
We say that the network if \emph{monotone} if every $\gamma_e$ is order-preserving ($x \ge y \Rightarrow \gamma_e(x) \ge \gamma_e(y)$), and we say that the
network is \emph{linear} if every $\gamma_e$ is linear (of the form $\gamma_e(x) = c_ex$ for some $c_e \in \mathbb{R}$).

For a fixed graph $\Gamma$ with boundary, a \emph{bijective conductance structure} on $\Gamma$ is a choice of the functions $\gamma_e$.
\end{definition}
The definition of compatible current and voltage functions and boundary values is the same as for conductance networks.

Note that we do not require the $\gamma_e$ to be continuous.
In the case
of monotone bijective networks, however, the $\gamma_e$ are all continuous, so by combining Theorems~\ref{dirichlet-problem} and Theorems~\ref{neumann-problem}
we have well-defined Dirichlet-to-Neumann and Neumann-to-Dirichlet maps.  In the general case of bijective networks, current may not be uniquely determined by
voltage, and vice versa, so we use a relation instead of a function:
\begin{definition}
If $\Gamma$ is a bijective network with $n$ boundary nodes, then the \emph{Dirichlet-Neumann relation} is the relation $\Lambda \subseteq \mathbb{R}^n \times
\mathbb{R}^n$ consisting of all pairs $(u,f)$, where $u$ is the boundary voltage of a voltage function $\phi$, $f$ is the boundary current of a current function
$\iota$, and $\phi$ and $\iota$ are compatible.
\end{definition}
The inverse problem is the problem of determining the conductance functions of a bijective network given the Dirichlet-Neumann relation and
the underlying graph.

\begin{definition}
A graph $\Gamma$ with boundary is \emph{strongly recoverable} if no two distinct bijective network structures on $\Gamma$ have the same Dirichlet-Neumann relation.
$\Gamma$ is \emph{weakly recoverable} if no two distinct monotone linear bijective network structures on $\Gamma$ have the same Dirichlet-Neumann relation.
\end{definition}
In other words, a graph is strongly recoverable if the inverse problem has a solution, in the sense that the conductance functions are uniquely determined
by the Dirichlet-Neumann relation.  A graph is weakly recoverable if this holds with the assumption that the conductance functions are monotone linear.
Weak recoverability is the type of recoverability considered in \cite{CIM} and \cite{ReseauxElectriques}.

\begin{definition}
Let $\Gamma$ be a graph with boundary.  We say that $\Gamma$ is \emph{circular planar} if it can be embedded in a disk $D$
with all boundary vertices on $\partial D$ and all interior vertices in the interior of $D$.
\end{definition}
Equivalently, $\Gamma$ is circular planar if the graph obtained by adding an extra vertex $v$ and connecting $v$ to every vertex
in $\partial \Gamma$ is planar.

Generalizing results of \cite{CIM} and \cite{ReseauxElectriques},
we will determine exactly which circular planar graphs are recoverable,
showing that for circular planar graphs, weak recoverability and strong recoverability are equivalent.
The techniques of \cite{CIM} used linear algebra and are no longer applicable in a non-linear setting, while the techniques of \cite{ReseauxElectriques}
will require modification because the Y-$\Delta$ transformation is no longer available.

\section{Covoltages}\label{sec:covoltages}
Working with voltages and currents is somewhat asymmetric because voltage functions live on vertices and are unconstrained,
while current functions live on edges and are constrained by Kirchhoff's Current Law.  To make matters more symmetric,
it will help to work with \emph{covoltages} rather than currents.

\begin{definition}
If $\Gamma$ is a circular planar graph with boundary (with a fixed planar embedding), then a \emph{covoltage function} on $\Gamma$
is an arbitrary real-valued function on the faces of $\Gamma$.  If $\psi$ is a covoltage function, the \emph{associated current
function} is the current function which assigns a value of $\psi(f_1) - \psi(f_2)$ to a directed edge with face $f_1$ on its left
and $f_2$ on its right.
\end{definition}
\begin{figure}
\centering
\def \svgwidth{5in}
\small
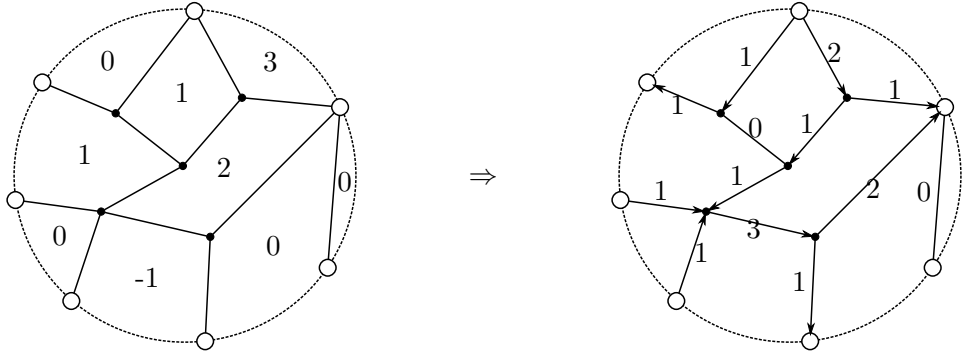
\caption{Covoltages on the left, and the corresponding currents on the right.  Arrows point in the direction of current flow.}
\label{fig10-covoltage}
\end{figure}
See Figure~\ref{fig10-covoltage} for an example.  We summarize the key facts about covoltage functions in the following lemma:
\begin{lemma}
~
\begin{itemize}
\item The associated current function of a covoltage function is a current function (it satisfies KCL).
\item If $\psi$ is a covoltage function, with associated current function $\iota$, then the boundary current
of $\iota$ at a boundary vertex $v$ is $\psi(f_1) - \psi(f_2)$, where $f_1$ and $f_2$ are the faces
on the left and right sides of $v$ (from $v$ facing into the graph).  See Figure~\ref{fig11-boundary-covoltage}
\item A covoltage function is determined by its associated current function up to a constant of integration.  In other words,
if two covoltage functions have the same associated current function, then their difference is constant.
\item Every current function is associated to a one covoltage function.
\end{itemize}
\end{lemma}
Only the last of these is nontrivial.  It follows from well-known facts about planar graph duality, and we leave its proof as an exercise to the reader.
\begin{figure}
\centering
\def \svgwidth{4in}
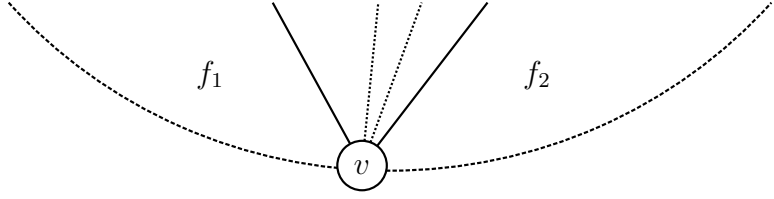
\caption{A boundary vertex $v$ and the two faces $f_1$ and $f_2$ on either side.}
\label{fig11-boundary-covoltage}
\end{figure}

From this lemma we see that covoltage functions and current functions are essentially equivalent.  We will generally
identify them, saying that a voltage function $\phi$ is compatible with a covoltage function $\psi$ if
$\phi$ is compatible with the associated current function of $\psi$.

More explicitly, a voltage function $\phi$ is compatible with a covoltage function $\psi$ if at every edge $e$ from vertex $v_1$
to $v_2$, with face $f_1$ on the left and $f_2$ on the right,
\[ \phi(v_1) - \phi(v_2) = \gamma_e(\psi(f_1) - \psi(f_2)).\]

From the lemma, we also see that boundary covoltages determine boundary currents, and
conversely boundary currents determine boundary covoltages up to an additive constant.  Consequently, the Neumann-Dirichlet
relation contains the same information as the \emph{voltage-covoltage relation} $\Xi$ which can be defined by
\[ \Xi = \{(\phi|_{\partial V},\psi|_{\partial F}):\textrm{ $\phi$ is a voltage function and $\psi$ is a compatible covoltage function}\}\]
where $\partial V$ and $\partial F$ are the sets of vertices and faces of $\Gamma$ along the boundary.

The relation between $\Xi$ and $\Lambda$ can be described explicitly as follows: let $v_1, f_1, v_2, f_2, \ldots, v_n, f_n$ be the boundary vertices
and faces of $\Gamma$, in counterclockwise order around the boundary of $\Gamma$.  Then we have
\[ (g,h) \in \Xi \iff (g(v_1),g(v_2),\ldots,g(v_n),h(f_n)-h(f_1),h(f_1)-h(f_2),\ldots,h(f_{n-1})-h(f_n)) \in \Lambda,\]
and conversely, $(x_1,x_2,\ldots,x_n,y_1,y_2,\ldots,y_n) \in \Lambda$ if and only if $y_1 + y_2 + \cdots + y_n = 0$ and $(g,h) \in \Xi$ where
$g(v_i) = x_i$ and $h$ is any function on the set of boundary faces with $h(f_n) - h(f_1) = y_1$, $h(f_1) - h(f_2) = y_2$, and so on.  Such
a function $h$ exists because $y_1 + y_2 + \cdots + y_n = 0$, and the choice of $h$ doesn't matter because $\Xi$ has the property
that $(g,h) \in \Xi$ if and only if $(g,h+\kappa) \in Xi$.

Because the voltage-covoltage relation contains the same information as the Neumann-Dirichlet relation, we can replace the Neumann-Dirichlet relation
with the voltage-covoltage relation in the definition of weak and strong recoverability.

\section{Medial Graphs}\label{sec:medial}
A graph's recoverability can be determined by considering an associated \emph{medial graph}:
\begin{definition}
A \emph{medial graph} is a piecewise-smooth Jordan curve $C \subseteq \mathbb{R}^2$, called the \emph{boundary curve}
and a collection of piecewise-smooth curves (\emph{geodesics}) in the inside of $C$, each of which is either closed, or has both endpoints on $C$, satisfying
the following properties:
\begin{itemize}
\item Every intersection of a geodesic with $C$ is transversal.
\item Every intersection between two geodesics or between a geodesic and itself is transversal.
\item There
are no triple intersection points of geodesics or the boundary curve.
\end{itemize}
\end{definition}
The geodesics partition the inside of the boundary curve into connected components called \emph{cells}.
Those along the boundary curve are called the \emph{boundary cells}.  The
\emph{interior vertices} of a medial graph are all the points of intersection between geodesics (including self-intersections).  The \emph{boundary vertices}
are the intersections of the geodesics with the boundary curve.

We draw the boundary curve with a dashed line.  The medial graph of Figure~\ref{fig2-simpler-medial} has three geodesics,
eleven cells, eight interior vertices, and four boundary vertices.
\begin{figure}
\centering
\def \svgwidth{2in}
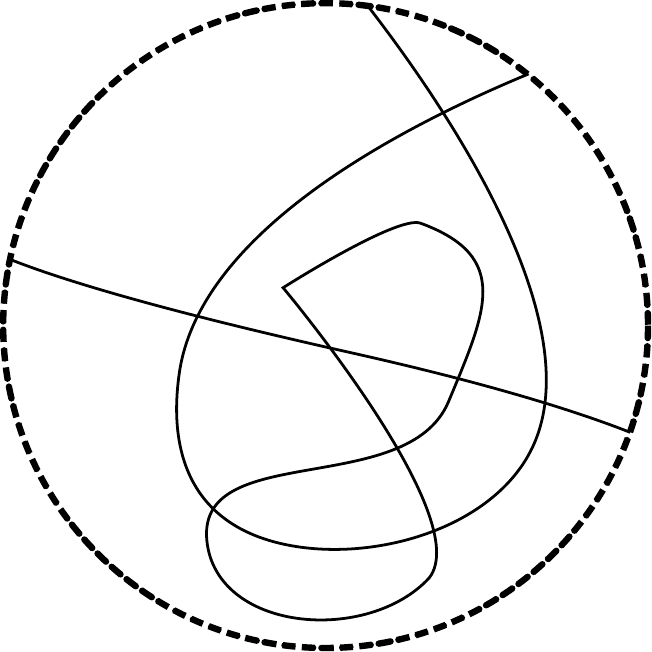
\caption{A sample medial graph.  This one is rather degenerate; one of the geodesics closes back up on itself.}
\label{fig2-simpler-medial}
\end{figure}

To any circular planar graph $\Gamma$ (with a fixed embedding in a disk $D$), we can construct a medial graph as follows: place a vertex along each
edge of $\Gamma$, and connect two vertices if they are on consecutive edges around one of the faces of $\Gamma$.  Also place two vertices along each
arc of $\partial D$ between consecutive boundary nodes of $\Gamma$, and connect these to the vertices on the closest edges.  Figure~\ref{fig3-cp-to-medial-non-inkscape}
shows this whole process.  The resulting
vertices and edges can then be decomposed as a union of geodesics in a unique way.  The result is a medial graph with at least $|\partial \Gamma|$ geodesics,
and boundary curve $\partial D$.
\begin{figure}
\centering
\def \svgwidth{5in}
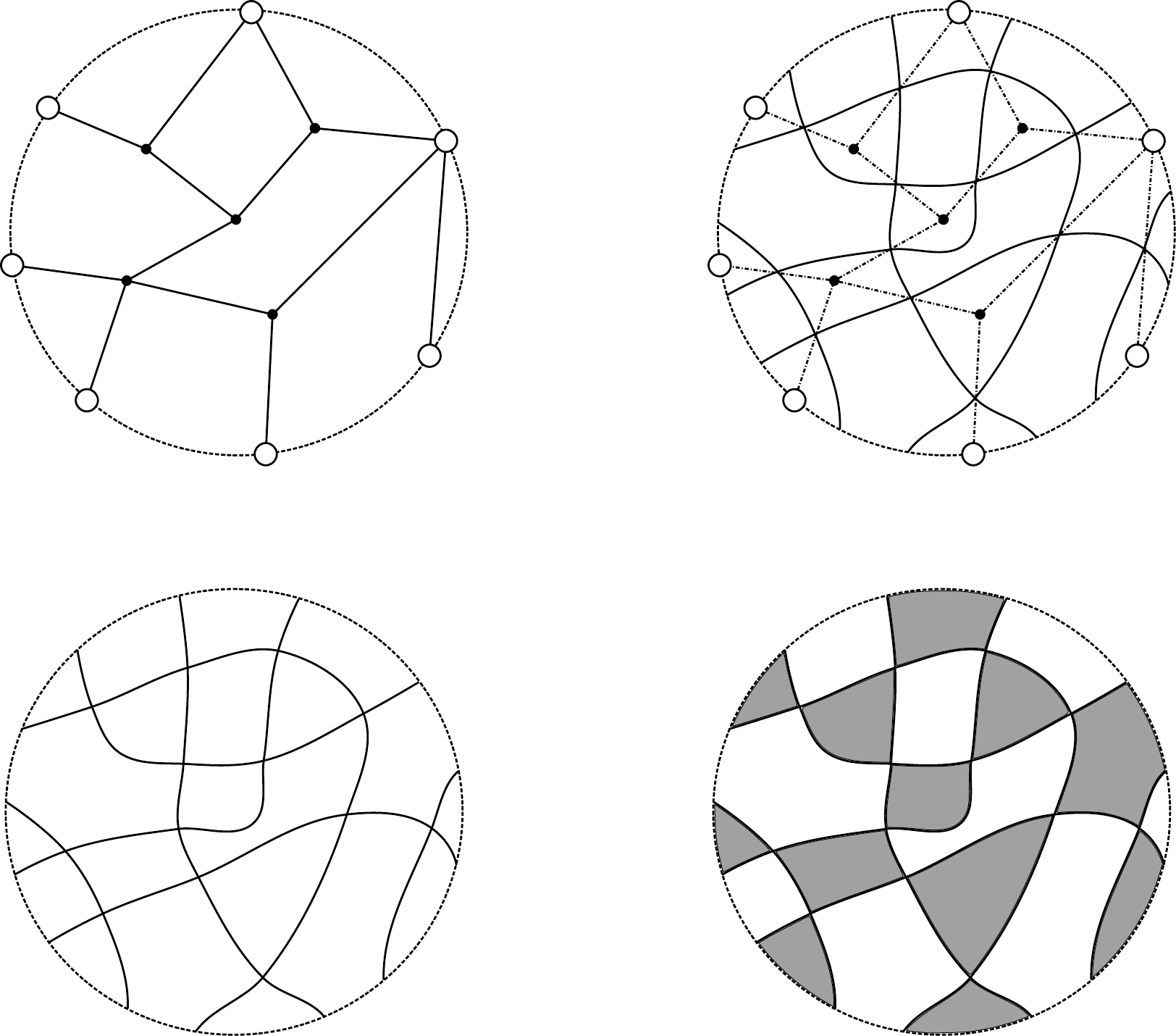
\caption{The process of converting a circular planar graph (top left) into a medial graph (bottom left).  The natural coloring of this medial graph
is shown on the bottom right.}
\label{fig3-cp-to-medial-non-inkscape}
\end{figure}

\begin{definition}
A \emph{coloring} of a medial graph is a function from the cells to the set $\{black, white\}$ such that directly adjacent cells have opposite colors.
Here we say that two cells are directly adjacent if they share an edge in common.
\end{definition}
Every medial graph has exactly two colorings.  A medial graph that comes from a circular planar graph $\Gamma$ has a canonical coloring, with the cells that
come from vertices colored black, and the cells that come from faces colored white.  The process of converting a circular planar graph into a colored medial graph
is almost reversible.  In particular, a circular planar graph $\Gamma$ is determined (up to isotopy) by its associated
colored medial graph.  However, some degenerate medial graphs, like those in Figure~\ref{fig4-impossible-medial-graphs}, do not come from circular planar graphs.
\begin{figure}
\centering
\def \svgwidth{4.5in}
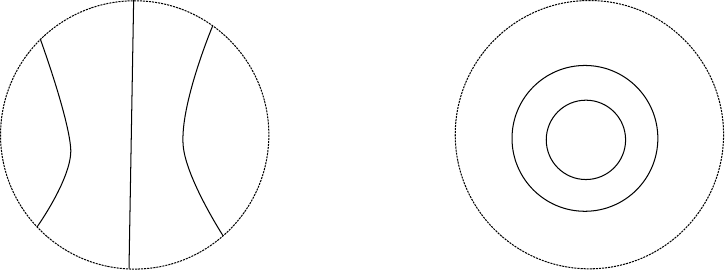
\caption{Medial graphs that cannot come from circular planar graphs (regardless of which way they are colored).}
\label{fig4-impossible-medial-graphs}
\end{figure}

Colored medial graphs therefore generalize circular planar graphs, and we can generalize the inverse boundary value problems to colored medial graphs in the following
way.
If $M$ is the medial graph of $\Gamma$, then the black cells of $M$ correspond to the vertices of $\Gamma$, while
the white cells of $M$ correspond to faces of $\Gamma$.  So a voltage function is equivalent to a function on the black
cells of $M$, while a covoltage function is equivalent to a function on the white cells of $M$.  The compatibility
condition then says that if $w$, $x$, $y$, $z$ are four cells that meet at an interior vertex corresponding to
a directed edge $e$, then
\[ \psi(z) - \psi(x) = \gamma_e(\phi(w) - \phi(y))\]
if $w$, $x$, $y$, $z$ are in counterclockwise order around $e$, and $w$ is the cell at the start of $e$, as in Figure~\ref{fig12-fourcell-setup}.
From the medial graph point of view, we might as well combine $\psi$ and $\phi$ into one function.
\begin{figure}
\centering
\def \svgwidth{4in}
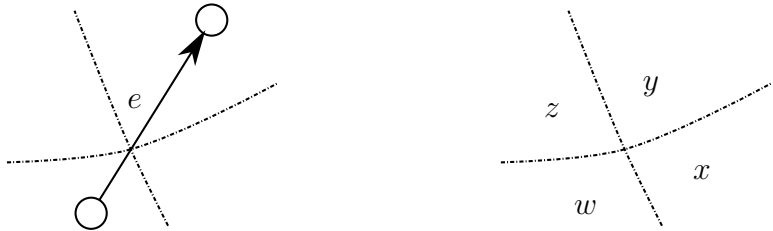
\caption{The four medial-graph cells surrounding an edge of the original graph.}
\label{fig12-fourcell-setup}
\end{figure}

\begin{remark}
To avoid a definitional nightmare, we will henceforth pretend that conductance functions $\gamma_e$ are always odd.
Equivalently, we pretend that $\gamma_e = \gamma_{\overline{e}}$.  None of the following results will use this assumption, but they are much harder
to state correctly without it!  For example, in the next definition, we would need to define something like an ``oriented interior vertex''
which is tricky to define correctly in the case where geodesics loop back on themselves.
\end{remark}

\begin{definition}
Let $M$ be a colored medial graph.  A \emph{conductance structure} on $M$ assigns to every interior vertex $v$ an odd bijection $\gamma_v : \mathbb{R} \to \mathbb{R}$.
A conductance structure is \emph{monotone} if every $\gamma_v$ is order-preserving, and \emph{linear} if every $\gamma_v$ is linear.
\end{definition}
When $M$ comes from a circular planar graph $\Gamma$, a conductance structure on $M$ is equivalent to a bijective conductance structure on $\Gamma$,
and the conditions of monotonicity and linearity are equivalent for $M$ and for $\Gamma$.

\begin{definition}
Let $M$ be a colored medial graph and $\gamma$ be a conductance structure on $M$.  A \emph{labelling} of $(M,\gamma)$ is a function $\phi$ from
the cells of $M$ to $\mathbb{R}$ such that for every interior vertex $v$ in $M$,
\begin{equation} \phi(z) - \phi(x) = \gamma_v(\phi(w) - \phi(y))\label{consistency-equation}\end{equation}
where $w$, $x$, $y$, $z$ are the cells in counterclockwise order around $v$, with $w$ and $y$ colored black.  The \emph{boundary
data} of a labelling $\phi$ is the restriction $\phi|_\partial$ of $\phi$ to the boundary cells.
The \emph{boundary relation} $\Xi_\gamma$ of $\gamma$ is the set
\[ \Xi_\gamma = \{\phi|_\partial ~:~ \phi\text{ is a labelling of } (M,\gamma)\}\]
of all
possible boundary data of labellings of $(M,\gamma)$.  A colored medial graph $M$ is \emph{strongly recoverable} if the function $\gamma \mapsto \Xi_\gamma$
is injective, and $M$ is \emph{weakly recoverable} if it is injective when restricted to monotone linear $\gamma$.
\end{definition}
Note that strong recoverability implies weak recoverability.
In the case where $M$ comes from a circular planar graph, it is not difficult to see that $\Xi_\gamma$ contains the same information as the voltage-covoltage
relation of $(\Gamma,\gamma)$, and in particular, $\Gamma$ is weakly (strongly) recoverable if and only if $M$ is.

If $M'$ is the colored medial graph obtained by reversing the colors of $M$, then $M$ is weakly (resp. strongly) recoverable if and only if $M'$ is weakly (resp. strongly)
recoverable.  In particular, recoverability
of a medial graph doesn't depend on the choice of the coloring.
We won't use this fact in what follows, so we leave the proof to the reader.  

\begin{definition}
A medial graph
is \emph{semicritical} if the following hold:
\begin{itemize}
\item No geodesic intersects itself
\item If $g_1$ and $g_2$ are geodesics, then $g_1$ and $g_2$ intersect at most once.
\end{itemize}
If in addition, every geodesic starts and ends on the boundary curve (rather than being a closed loop), we say that the medial graph is \emph{critical}.
\end{definition}
Critical medial graphs are essentially equivalent to \emph{simple pseudoline arrangements}.

The main result we are working towards is the following:
\begin{theorem}\label{main-result}
Let $M$ be a medial graph.  Then the following are equivalent:
\begin{itemize}
\item $M$ is strongly recoverable.
\item $M$ is weakly recoverable.
\item $M$ is semicritical.
\end{itemize}
In particular, this gives a criterion for determining whether a circular planar graph is weakly or strongly recoverable.
\end{theorem}
The equivalence of weak recoverability and semicriticality for medial graphs coming from circular planar graphs was shown in \cite{CIM} and \cite{ReseauxElectriques} by arguments involving
linear algebra and determinantal identities.  It therefore comes as a surprise that this remains true in the highly nonlinear cases considered here.
It is also interesting that for circular planar graphs, weak and strong
recoverability are equivalent.  It is unclear whether this remains true for non-planar graphs.

\section{Lenses, Boundary Triangles, and Motions}\label{sec:linear}
In this section, we review facts about medial graphs from \cite{CandM}, and show why a weakly recoverable medial graph is semicritical.
Almost all of the proofs of this section are based on those in \cite{CandM}.  We use unusual terminology, however.

\begin{definition}
A \emph{boundary digon} or \emph{boundary triangle} is a simply-connected boundary cell with two or three sides, respectively.  If $g$ is a geodesic which
does not intersect the boundary curve or any other geodesics and which is a simple closed curve, then we call the set of cells inside $g$ a \emph{circle}.
If there is just one cell, we call it an \emph{empty circle}.
\end{definition}

\begin{definition}
If $M$ is a critical medial graph and $g$ is a geodesic, then since $M$ is critical, $g$ divides $M$ into two connected components, by the Jordan curve theorem.
We call such a connected component a \emph{side}.  If $g$ and $h$ are two geodesics, then $g \cup h$ divides $M$ into four pieces, which we call \emph{wedges}.
The boundary of any wedge consists of a segment of the boundary curve, which we call the \emph{base}, and two segments of $g$ and $h$, which we call the \emph{legs}.
We call the intersection of $g$ and $h$ the \emph{apex} of a wedge.  We say that a wedge $W$ is a \emph{semiboundary wedge} if at least one of the two legs contains no
interior vertices other than the apex.  A \emph{boundary zone} is one of the following:
\begin{itemize}
\item A semiboundary wedge
\item The side of a geodesic which intersects no other geodesics. 
\end{itemize}
\end{definition}

\begin{theorem}\label{boundary-triangles-exist}
Let $M$ be a critical medial graph.  If $M$ has more than one cell, then $M$ has at least one boundary digon, or at least three boundary triangles.
\end{theorem}
\begin{proof}
We first prove the following:
\begin{lemma}
Every boundary zone contains a boundary triangle or a boundary digon.
\end{lemma}
\begin{proof}
Let $Z$ be a boundary zone.  Without loss of generality, $Z$ contains no smaller boundary zones.  I claim that no geodesic crosses the boundary of $Z$.
This is trivial if $Z$ is the side of a geodesic which intersects
no other geodesics.  In the other case, $Z$ is a semiboundary wedge.  Say $Z$ has apex $a$ and legs $ab$ and $ac$.  Without loss of generality, no geodesics
cross $ab$, and it remains to show that no geodesics cross $ac$.  Suppose a geodesic $g$ crosses $ac$ at an interior vertex $d$.  Without loss of generality, $d$ is
as close to $c$ as possible, and in particular, there are no interior vertices between $c$ and $d$.  After $g$ enters $Z$ at $d$, it must do one of three things:
\begin{itemize}
\item Exit $Z$ by passing through $ab$.
\item Exit $Z$ by passing through $ac$.
\item End on the boundary curve.
\end{itemize}
The first of these is impossible because there are no interior vertices on $ab$.  The second is impossible because it would mean that $g$ crosses $ac$ in two places.
The third is impossible because it would mean that there was a smaller semiboundary wedge inside $Z$, as in Figure~\ref{focusing-wedge}, contradicting the choice of $Z$.

So no geodesic enters or exits $Z$.  I claim that there are in fact no geodesics in $Z$.  Let $g$ be any geodesic in $Z$.  If $g$ intersects no other geodesics, then
one side of $g$ is a smaller boundary zone inside $Z$, a contradiction.  Otherwise, there is at least one interior vertex along $g$.  Let $e$ be one which is as close
to the boundary curve as possible, as in Figure~\ref{smaller-wedges}.  Then $e$ is the apex of at least two semiboundary wedges contained inside $Z$, a contradiction.

Therefore, $Z$ contains no geodesics in its interior. It follows that $Z$ must be a single cell.  Therefore $Z$ is either a boundary triangle or a boundary digon.
\end{proof}
We next, show the existence of at least one boundary triangle or boundary digon.  By the lemma, it suffices to find at least one boundary zone.
Since $M$ has more than one cell, it has more than zero geodesics.  Let $g$ be a geodesic.  If $g$ intersects no other geodesics, then both sides of $g$
are boundary zones, so we have at least one boundary digon or boundary triangle.  Otherwise, let $v$ be a boundary vertex of $g$ which is as close as possible
to one of the endpoints of $g$.  If $h$ crosses $g$ at $v$, then half the wedges formed by $g$ and $h$ are semiboundary wedges, so at least one boundary digon
or boundary triangle exists.

If a boundary digon exists, then we are done.  Otherwise, we have a boundary triangle.  This boundary triangle is in fact a wedge formed by two geodesics
$g$ and $h$.  Two of the other wedges formed by $g$ and $h$ are semiboundary wedges, so this gives two more boundary triangles or boundary digons.  This completes
the proof of Theorem~\ref{boundary-triangles-exist}.
\end{proof}
\begin{figure}
\centering
\def \svgwidth{4in}
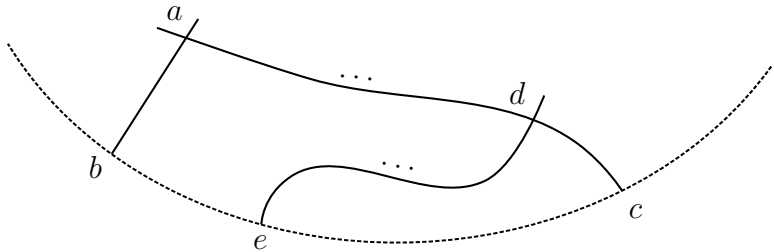
\caption{If there are any interior vertices along leg $ac$, then $abc$ is not a minimal semiboundary wedge.}
\label{focusing-wedge}
\end{figure}
\begin{figure}
\centering
\def \svgwidth{4in}
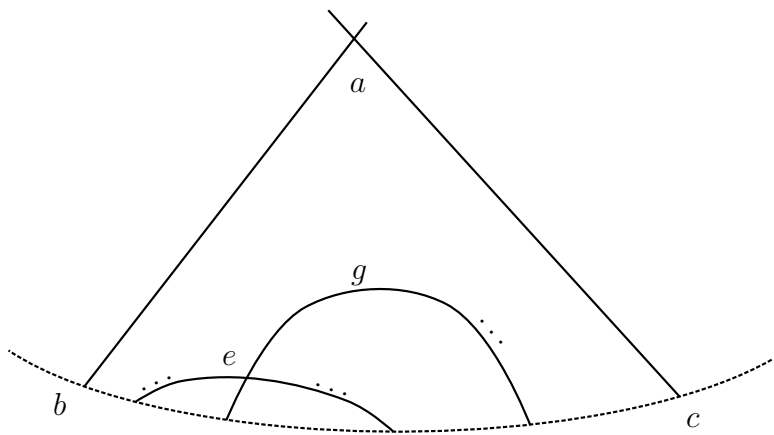
\caption{If any interior vertex $e$ exists along $g$, then $e$ is the apex of at least two semiboundary wedges inside $abc$.}
\label{smaller-wedges}
\end{figure}

\begin{definition}
A medial graph is \emph{circleless} if it has no empty circles.
\end{definition}

\begin{theorem}\label{kind-of-obvious}
A semicritical circleless graph $M$ is critical.
\end{theorem}
\begin{proof}
Given the definition of critical and semicritical, it suffices to show that if $M$ is circleless, then it has no circles.
Let $C$ be a circle which minimizes the number of cells inside $C$.  Since all the intersections of geodesics are transversal, if $g$ is any geodesic,
then $g$ intersects (the boundary of) $C$ an even number of times.  This can only happen in a semicritical graph if $g$ intersects (the boundary of) $C$ exactly zero
times.  It follows that every geodesic in $M$ is either inside $C$ or outside $C$.  But every geodesic inside $C$ must itself be a circle, contradicting the choice of
$C$.  Therefore $C$ is an empty circle, and $M$ is not circleless.
\end{proof}

\begin{definition}
If $g$ is a geodesic, we can think of $g$ as either a continuous function on $[0,1]$ or a continuous function on $\mathbb{R}$ with period 1.
If $[a,b]$ is a subinterval of the domain of $g$, and $g$ is injective on both half-open intervals $(a,b]$ and $[a,b)$, we call the curve $g([a,b])$
a \emph{geodesic segment with endpoints $g(a)$ and $g(b)$}.
\end{definition}

\begin{definition}
Let $v$ be an interior vertex.  A \emph{1-lens with pole at $v$} is the interior of a Jordan curve that is a geodesic segment with endpoints $v$ and  $v$.
Let $v$ and $w$ be interior vertices.  A \emph{2-lens with poles at $v$ and $w$} is the interior of a Jordan curve that is a union of two geodesic segments,
each having endpoints at $v$ and $w$.  For $n = 1, 2$, we say that an $n$-lens is \emph{sharp} if for every pole $v$, only one of the four cells that meet at $v$ is
inside the $n$-lens.  We say that an $n$-lens is \emph{empty} if it contains only one cell.
\end{definition}
Because every medial graph can be colored, no cell can be on both sides of a geodesic.  From this it is easy to see that empty lenses must be sharp lenses.
See Figure~\ref{fig6-lenses} for examples of lenses.
\begin{figure}
\centering
\def \svgwidth{5in}
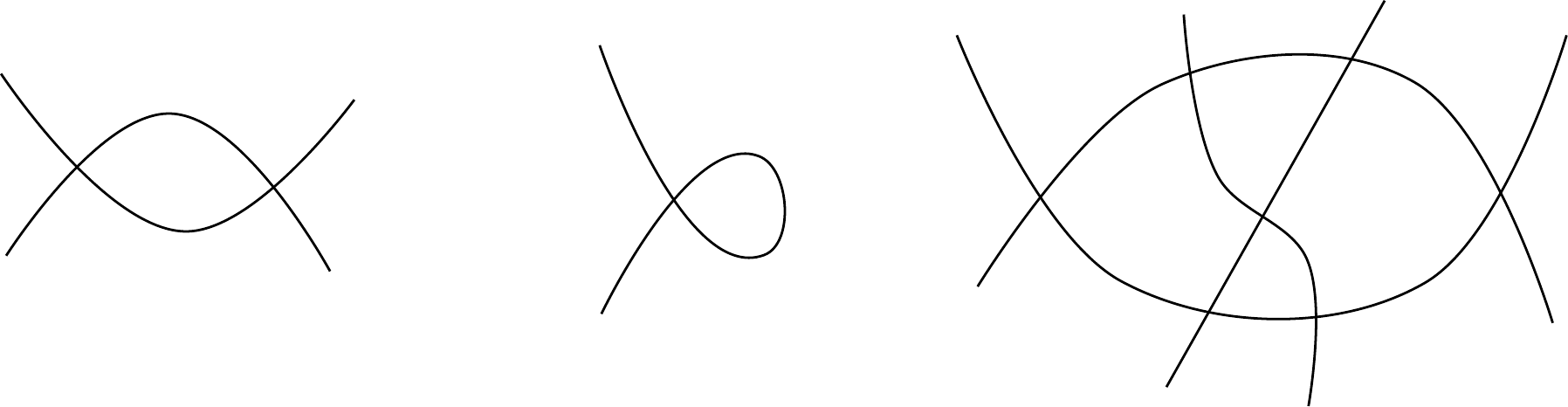
\caption{An empty 2-lens, an empty 1-lens, and a nonempty sharp 2-lens.}
\label{fig6-lenses}
\end{figure}

\begin{theorem}\label{lenses-exist}
If $M$ is a medial graph which is not semicritical, then $M$ must contain a sharp 1-lens or sharp 2-lens.
\end{theorem}
\begin{proof}
Since $M$ is not semicritical, we either have a geodesic which self-intersects, or two geodesics which intersect too many times.  First suppose
that some geodesic $g$ intersects itself.  Then we have $g(s) = g(t)$ for some $s \ne t$, and in the case where $g$ is a closed curve and we think of
$g$ as a function on $\mathbb{R}/\mathbb{Z}$, we want $s - b \notin \mathbb{Z}$.  Choose a pair $\{s,t\}$ for which all this holds, minimizing $|s - t|$.
Without loss of generality $s < t$.  I claim that $g([s,t])$ is a geodesic segment.  In the periodic case, clearly $|s - t| < 1$, or else we could replace
$s$ with $s + 1$ and make $|s - t|$ smaller.  If $g$ fails to be injective on $(s,t]$, then there exist $s < r_1 < r_2 \le t$ with $g(r_1) = g(r_2)$
and $|r_1 - r_2| < |s - t|$ (and $r_1 - r_2 \notin \mathbb{Z}$ in the periodic case), contradicting the choice of $s, t$.  A similar argument shows
that $g$ is injective on $[s,t)$.  So $g([s,t])$ is a geodesic segment with endpoints at $g(s)$ and $g(t)$.
In particular, it is a 1-lens, with pole at $v = g(s) = g(t)$.

If this 1-lens fails to be sharp, then some geodesic $h$ at $v$ enters into the 1-lens.  It eventually exits the 1-lens.  If it does so without crossing itself,
then we have a smaller 2-lens or 1-lens, as in Figure~\ref{no-escape-1}.
Otherwise, we can repeat all of the above argument on $h$ (or rather, on $h$ restricted to the part between $v$ and
where $h$ first exits the 1-lens), and get a smaller 1-lens, and continue inductively, until we get a sharp 1-lens.

Next, suppose that no geodesic intersects itself.  Then $M$ can only fail to be semicritical if two geodesics $g$ and $h$ intersect more than once.  So there
exist $s, t$ in the domain of $h$ with $h(s)$ and $h(t)$ in the range of $g$, and $s \ne t$.  As before, we can take $s - t \notin \mathbb{Z}$ in the periodic
case.  Choose such an $\{s,t\}$ with $|s - t|$ minimized.  Then as before, $|s - t| < 1$ in the periodic case.  By choice of $\{s, t\}$,
$h(r)$ cannot lie on $g$ for any $r$ between $s$ and $t$.  It follows immediately that the geodesic segments between $h(s)$ and $h(t)$ along $h$ and $g$
form a 2-lens, with poles at $v = h(s)$ and $w = h(t)$.  To see that this 2-lens is sharp, suppose that more than one of the cells around $v$ lies inside the 2-lens.
Since $h \ne g$ and the intersection at $v$ is transversal and the boundary of the 2-lens is a Jordan curve, exactly three of the cells lie inside, and we have
the configuration of Figure~\ref{no-escape-2}.
Then $g$ continues on from $v$ into the interior of the 2-lens.  It has no way of escaping, however: it cannot intersect itself
(by assumption), and it cannot intersect $h([s,t])$ (by choice of $\{s,t\}$).  This is a contradiction.
\end{proof}
\begin{figure}
\centering
\def \svgwidth{4in}
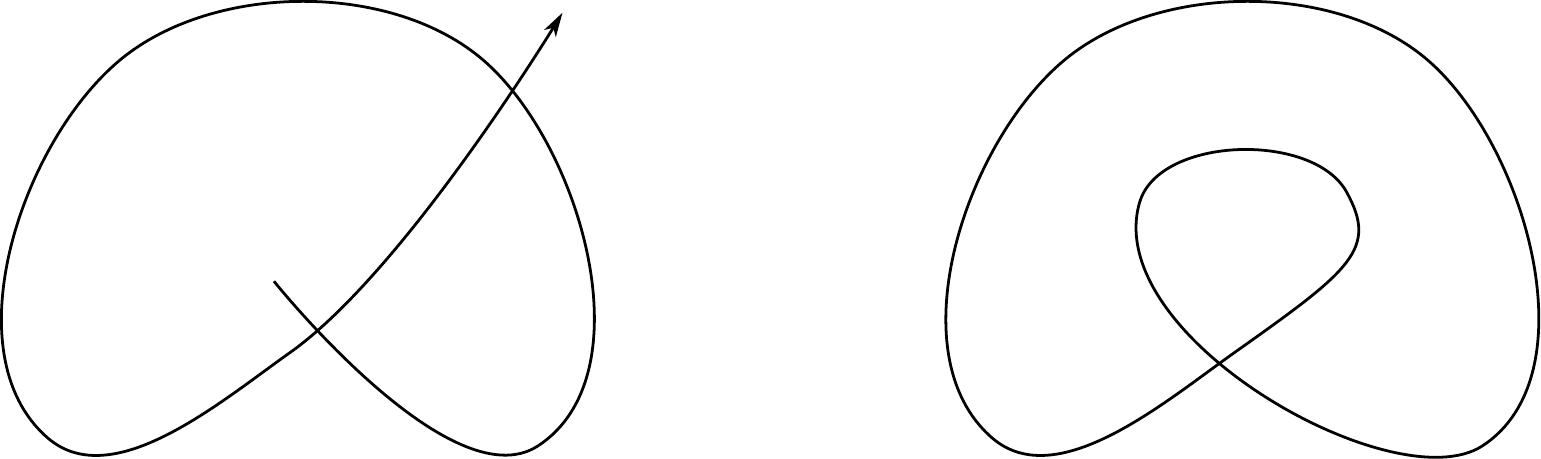
\caption{A non-sharp 1-lens must contain a smaller 2-lens or 1-lens.  A third possibility, not pictured here, is that the geodesic crosses itself
before entering and exiting the original 1-lens.  Then one repeats the argument of Theorem~\ref{lenses-exist} inductively.}
\label{no-escape-1}
\end{figure}
\begin{figure}
\centering
\def \svgwidth{2.5in}
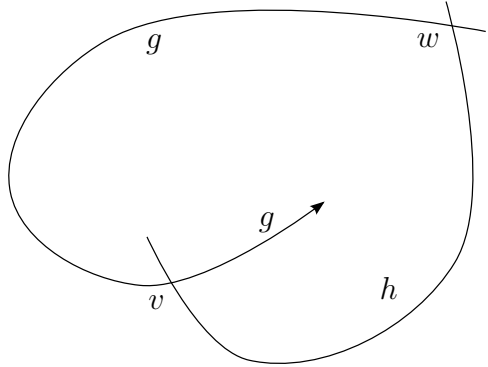
\caption{If sharpness fails at pole $v$, then geodesic $g$ has no means of exiting the 2-lens.  (The pole at $w$ needn't be sharp in this arrangement.)}
\label{no-escape-2}
\end{figure}

\begin{definition}
A \emph{Y-$\Delta$ motion} is the transformation shown in Figure~\ref{fig9-medial-graph-move}.
\end{definition}
If our medial graphs come from circular planar networks, this is nothing but a standard Y-$\Delta$ transformation as in Figure~\ref{fig77-y-delta}.
\begin{figure}
\centering
\def \svgwidth{4in}
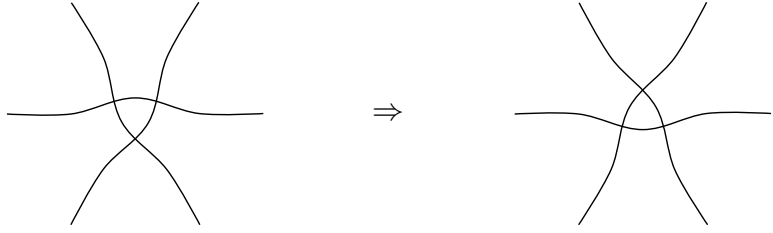
\caption{A Y-$\Delta$ motion in a medial graph.}
\label{fig9-medial-graph-move}
\end{figure}
\begin{figure}
\centering
\def \svgwidth{4in}
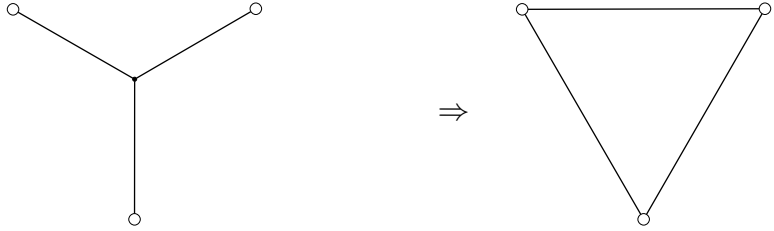
\caption{A Y-$\Delta$ transformation.}
\label{fig77-y-delta}
\end{figure}

\begin{theorem}\label{emptying-lenses}
If $M$ is a medial graph which is not semicritical, then we can apply some combination of the following operations to reach a medial graph which contains
an empty 1-lens or 2-lens:
\begin{itemize}
\item A Y-$\Delta$ motion
\item Removal of an empty circle.
\end{itemize} 
\end{theorem}
\begin{proof}
By the previous theorem, there is at least one sharp lens.  We prove by induction on $n$ that if $M$ is a medial graph containing a sharp lens which contains $n$ cells,
then $M$ can be converted into a medial graph which contains an empty lens.  The base case where $n = 1$ is trivial, because then $M$ already has an empty lens.
Suppose $n > 1$, and we have a sharp lens $L$ with $n$ cells.  If there are any smaller sharp lenses inside $L$, we are done by induction.  Suppose that $L$ contains
no smaller sharp lenses.  Since the boundary of $L$ is a Jordan curve,
we can make a new medial graph $M'$ from the inside of $L$, with the boundary of $L$ as its boundary curve.  I claim that $M'$
is semicritical.  Otherwise, it contains a sharp lens, by the previous theorem, contradicting the fact that $L$ contains no smaller sharp lenses.
If $M'$ contains an empty circle, then we can remove it, and continue by induction.  So we can suppose that $M'$ is circleless.  Then by Theorem~\ref{kind-of-obvious},
$M'$ is critical.  Therefore, it contains at least three boundary triangles, or at least one boundary digon.  If there are at least three boundary triangles,
then at least one of them does not touch a pole of $L$.  Therefore it corresponds to an actual triangle in $M$, and we can carry out a Y-$\Delta$ motion,
pushing the intersection
outside of $L$, while preserving the lens $L$.  This decreases the number of cells in $L$, so we can continue by induction.  Finally, suppose that there is a boundary
digon.  If it does not touch one of the poles of $L$, it corresponds to an empty lens, and we are done.  Otherwise, it corresponds to a triangle in $M$,
and we can push it outside of $L$ by a Y-$\Delta$ motion, and then continue by induction.
\end{proof}

In the monotone linear case, a conductance structure on a medial graph is merely a function that assigns a conductance $c_v > 0$ to each interior vertex $v$
of $M$.  The conductance functions are of the form $\gamma_v(x) = c_v x$.
\begin{lemma}
The following transformations do not effect the boundary relation of a medial graph.  We assume that the vertices involved in the transformations have
monotone linear conductance functions, but the rest of the medial graph need not.
\begin{enumerate}
\item Removing a boundary digon.
\item Removing an empty circle.
\item A motion as in Figure~\ref{fig99-medial-graph-move}.  Here, the new conductances should be given by
\begin{align*}
c_{12} & = \frac{c_1 c_2}{c_1 + c_2 + c_3} \\
c_{13} & = \frac{c_1 c_3}{c_1 + c_2 + c_3} \\
c_{23} & = \frac{c_2 c_3}{c_1 + c_2 + c_3}
\end{align*}
in one direction, and
\begin{align*}
c_1 & = \frac{c_{12} c_{13} + c_{12} c_{23} + c_{13} c_{23}}{c_{23}} \\
c_2 & = \frac{c_{12} c_{13} + c_{12} c_{23} + c_{13} c_{23}}{c_{13}} \\
c_3 & = \frac{c_{12} c_{13} + c_{12} c_{23} + c_{13} c_{23}}{c_{12}}
\end{align*}
in the other direction.
\item Removing an immediate self-loop.
\item Replacing an empty lens with a single intersection, as in Figure~\ref{fig9-point-5-series-parallel}.
In the top row of Figure~\ref{fig9-point-5-series-parallel}, the conductance of the new vertex should be
\[ \frac{c_1 c_2}{c_1 + c_2}\]
while in the bottom row it should be
\[ c_1 + c_2 \]
\end{enumerate}
\end{lemma}
\begin{figure}
\centering
\def \svgwidth{4in}
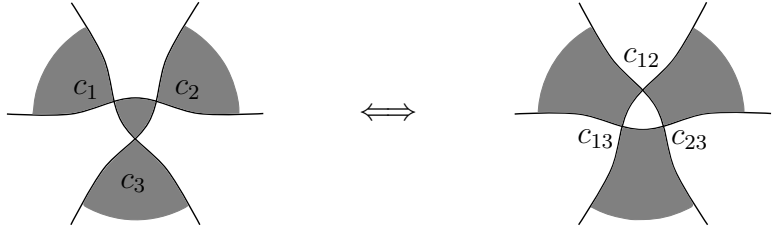
\caption{A Y-$\Delta$ motion in a colored medial graph.}
\label{fig99-medial-graph-move}
\end{figure}
\begin{figure}
\centering
\def \svgwidth{4in}
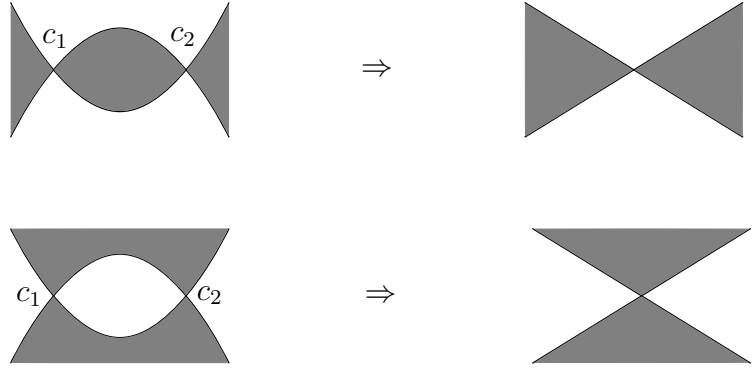
\caption{Replacing an empty lens with a single intersection.}
\label{fig9-point-5-series-parallel}
\end{figure}
\begin{proof}
These all follow by direct calculation.  The third rule is just the standard Y-$\Delta$ transformation, while the fifth consists of the standard
rules for resistors in series or parallel.
\end{proof}
The Y-$\Delta$ transformation is reversible, while the process of replacing a series or parallel configuration with a single resistor throw away information.
This immediately yields the following conclusion:
\begin{corollary}\label{transformations-and-recoverability}~
\begin{itemize}
\item Adding or removing a boundary digon does not affect weak or strong recoverability.
\item Adding or removing an empty circle does not affect weak or strong recoverability.
\item A Y-$\Delta$ motion does not affect weak recoverability.
\item Any medial graph containing an immediate self-loop or an empty lens is not weakly or strongly recoverable.
\end{itemize}
\end{corollary}

This yields half of Theorem~\ref{main-result}.
\begin{corollary}
If $M$ is a medial graph and $M$ is not semicritical, then $M$ is not weakly recoverable.
\end{corollary}
\begin{proof}
Combine Corollary~\ref{transformations-and-recoverability} and Theorem~\ref{emptying-lenses}.
\end{proof}

For the other direction, we will use the following property of boundary triangles.
\begin{theorem}
Let $M$ be a medial graph with a boundary triangle $abc$, and let $M'$ be the medial graph obtained by \emph{uncrossing} the boundary
triangle, as in Figure~\ref{fig23-uncrossing}.
Let $a$ be the apex of $abc$.  If we fix the conductance function $\gamma_a$, then the boundary relation of $M'$ is determined
by the boundary relation of $M$.
\end{theorem}
\begin{proof}
Let $w,x,y,z$ be the four cells around $a$, as in Figure~\ref{fig23-uncrossing}.
By (\ref{consistency-equation}), if $\phi$ is any labelling of $M'$, then $\phi$ extends to a labelling of $M$
uniquely by setting
\[ \phi(w) = \phi(y) \pm \gamma_a(\phi(x) - \phi(y))\]
where the $\pm$ depends on the coloring.  It follows that if $\Xi$ and $\Xi'$ are the boundary relations of $M$ and $M'$, then
\[ (\ldots,p,q,r,\ldots) \in \Xi' \iff (\ldots,p,q \pm \gamma_a(p-r), r, \ldots) \in \Xi,\]
so $\Xi$ and $\Xi'$ determine each other.
\end{proof}
\begin{figure}
\centering
\def \svgwidth{3.5in}
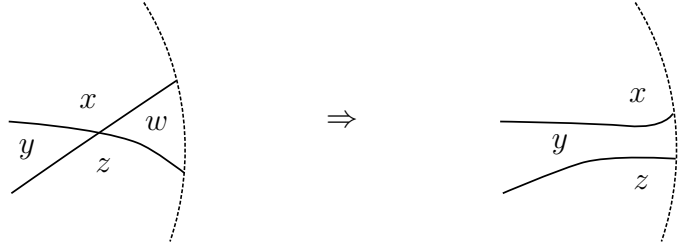
\caption{``Uncrossing'' a boundary triangle.}
\label{fig23-uncrossing}
\end{figure}

In light of this, we can reduce Theorem~\ref{main-result} to one condition:
\begin{theorem}\label{conditional}
Suppose the following statement is true: if $M$ is any critical colored medial graph, and $abc$ is a boundary triangle of $M$, with apex $a$, then
$\gamma_a$ is determined by the boundary relation of $M$.  Then Theorem~\ref{main-result} is true.
\end{theorem}
\begin{proof}
We have already seen that weak recoverability implies semicriticality.  Clearly, strong recoverability implies weak recoverability.  It remains to show
that semicriticality implies strong recoverability.  Let $M$ be a semicritical medial graph.  Suppose $M$ is not strongly recoverable, and take $M$ to have as few cells
as possible.  By Theorem~\ref{transformations-and-recoverability}, we can remove empty circles from $M$ without affecting strong recoverability.  Therefore,
$M$ is circss, and by Theorem~\ref{kind-of-obvious}, it is critical.  Similarly, $M$ cannot have a boundary digon, by Theorem~\ref{transformations-and-recoverability}.
If $M$ has no geodesics, then $M$ is strongly recoverable because the inverse boundary problem requires the recovery of no information.  Otherwise, $M$ must have
a boundary triangle by Theorem~\ref{boundary-triangles-exist}.  Let $M'$ be the medial graph obtained by uncrossing the boundary triangle.  By assumption, we can
find the conductance function at the apex of the boundary triangle, and by the previous theorem, this in turn yields the boundary relation of $M'$.
Uncrossing a boundary triangle does not break criticality, so $M'$ is critical, and by choice of $M$ it is therefore strongly recoverable.  So given the boundary
relation of $M'$, we can find all the remaining conductance functions of $M$.  Therefore $M$ is strongly recoverable, a contradiction.
\end{proof}

It remains to show how to recover boundary the conductance functions of boundary triangles.  We will do this by setting up special boundary value problems
which allow us to read off the conductance function at a boundary triangle.\footnote{In \cite{CIM}, conductances of boundary triangles are recovered
using linear algebra.  Since we allow nonlinear conductances, this approach won't work here.}
The main difficulty will be showing that our boundary value problem is well-posed,
having a unique solution.

\section{The Propagation of Information}\label{sec:info-prop}
To accomplish this, we will need a way of modelling the propagation of information through a medial graph.  We begin
by defining partial labellings of medial graphs.
\begin{definition}
If $M$ is a fixed medial graph, a \emph{cellset} is a set of cells in $M$.  If $M$ is the medial graph associated
to a bijective network $\Gamma$, and $S$ is a cellset, then a \emph{labelling} of $S$ is a function $\phi:S \to \mathbb{R}$
which satisfies (\ref{consistency-equation}) whenever $w,x,y,z \in S$.
\end{definition}
Note that the notion of a labelling depends on the conductance functions.

Next we turn to combinatorial definitions that depend solely on $M$, not even on its coloring.  These will
be our model for the propagation of information.  For motivation, see Theorem~\ref{safe-sets-work} below.
\begin{definition}
Let $S$ be a cellset.  Then define $\rank(S)$ to be $|S| - v(S)$ where $|S|$ is the number
of cells in $S$, and $v(S)$ is the number of interior vertices for which all four surrounding cells are in $S$.
\end{definition}
\begin{figure}
\centering
\def \svgwidth{4in}
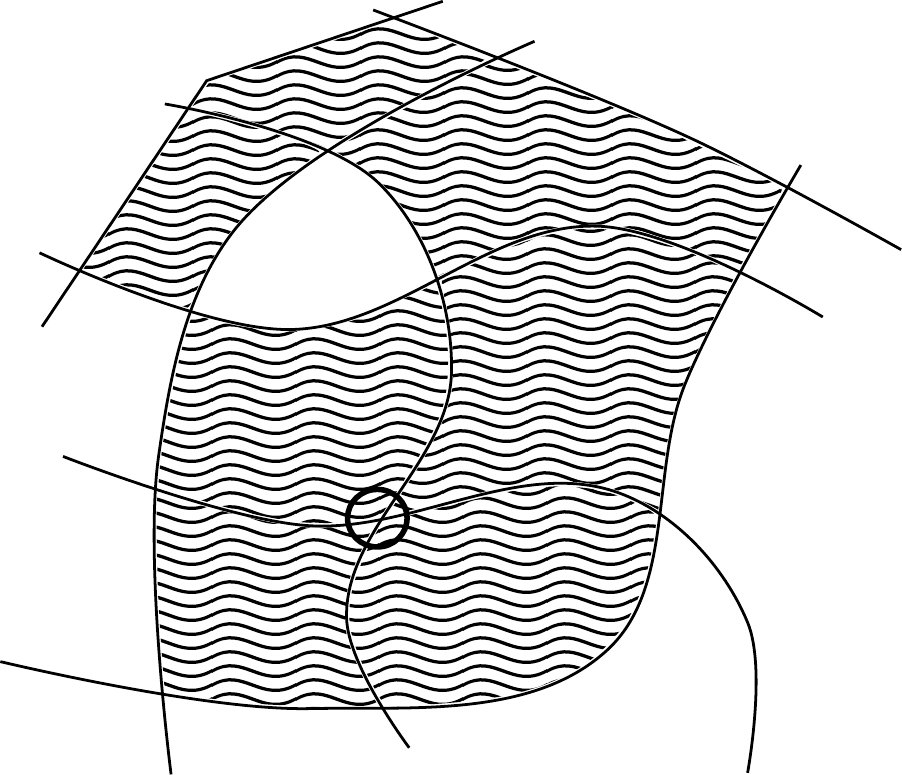
\caption{The wavy cellset has rank six, because it contains seven cells, and there is one interior vertex completely surrounded
by $S$ (the circled one).}
\label{fig14-rank-example}
\end{figure}
The rank of $S$ estimates the number of degrees of freedom of a labelling of $S$, since it counts the number
of cells (the number of variables) minus the number of constraining equations.  See Figure~\ref{fig14-rank-example} for an example.
\begin{definition}
If $S$ is a cellset and $x$ is a cell not in $S$, then we say that $S \cup \{x\}$ is a \emph{simple extension}
of $S$ if there is an interior vertex $v$ of $M$ for which three of the cells around $v$ are in $S$ and $x$ is the fourth.
If there does not exist another interior vertex $v'$ with the same property, then we say that $S \cup \{x\}$ is
a \emph{safe simple extension}.
\end{definition}
\begin{figure}
\centering
\def \svgwidth{5in}
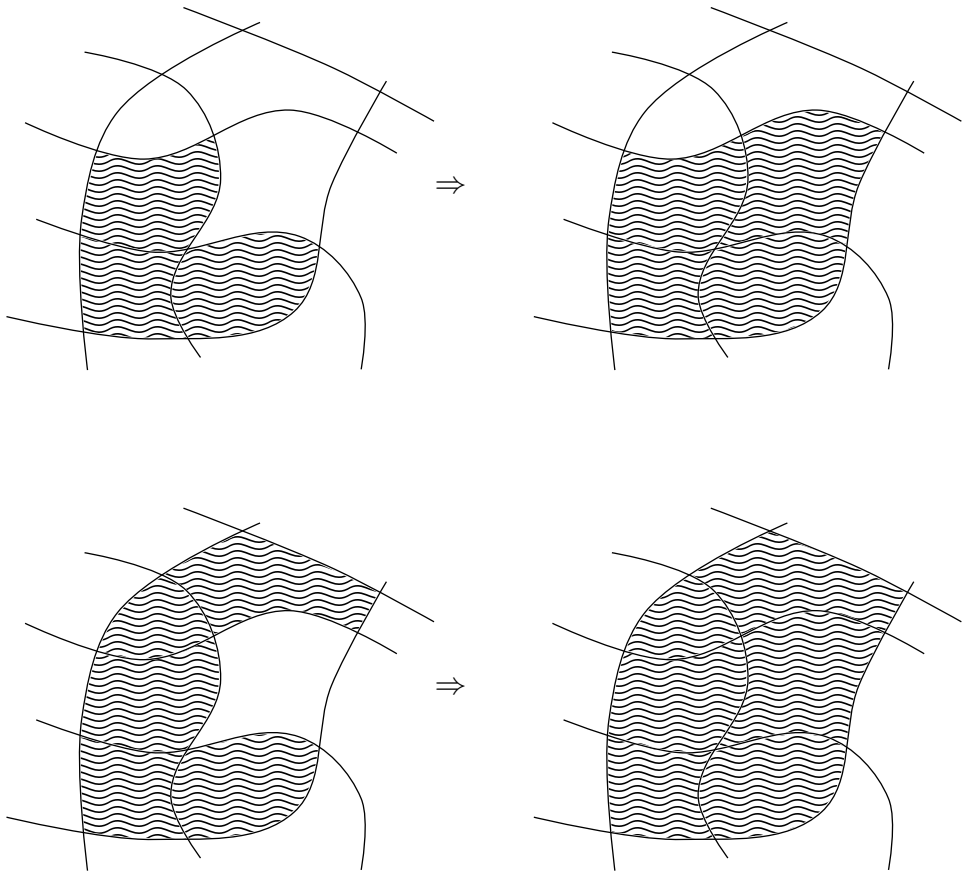
\caption{The top right is a safe simple extension of the top left, while the bottom right is a (non-safe) simple extension of the bottom left.}
\label{fig13-simple-extensions}
\end{figure}
\begin{lemma}\label{rank-and-extension}
Let $S$ be a cellset and $x$ be a cell not in $S$.  Then $\rank(S \cup \{x\}) \le \rank(S)$ if $S \cup \{x\}$
is a simple extension of $S$, and $\rank(S \cup \{x\}) = \rank(S) + 1$ otherwise.  Moreover,
$\rank(S \cup \{x\}) = \rank(S)$ if and only if $S \cup \{x\}$ is a safe simple extension of $S$.
\end{lemma}
\begin{proof}
Obviously $|S \cup \{x\}| = |S| + 1$.  So letting $v(T) = \rank(S) - |S|$ as above, we need to show that
\begin{itemize}
\item If $S \cup \{x\}$ is not a simple extension of $S$, then $v(S \cup \{x\}) = v(S)$
\item If $S \cup \{x\}$ is a safe simple extension of $S$, then $v(S \cup \{x\})  = v(S) + 1$.
\item If $S \cup \{x\}$ is a simple extension of $S$ that is not safe, then $v(S \cup \{x\}) > v(S) + 1$.
\end{itemize}
But these are all obvious from the definitions of simple extensions and safe simple extensions.
\end{proof}
\begin{definition}
A cellset $S$ is \emph{closed} if it has the property that for every interior vertex $v$ of $M$,
if three of the four cells around $v$ are in $S$, then so is the fourth.  The smallest closed set
containing $S$ is the \emph{closure of $S$}, denoted $\overline{S}$.
\end{definition}
\begin{figure}
\centering
\def \svgwidth{2in}
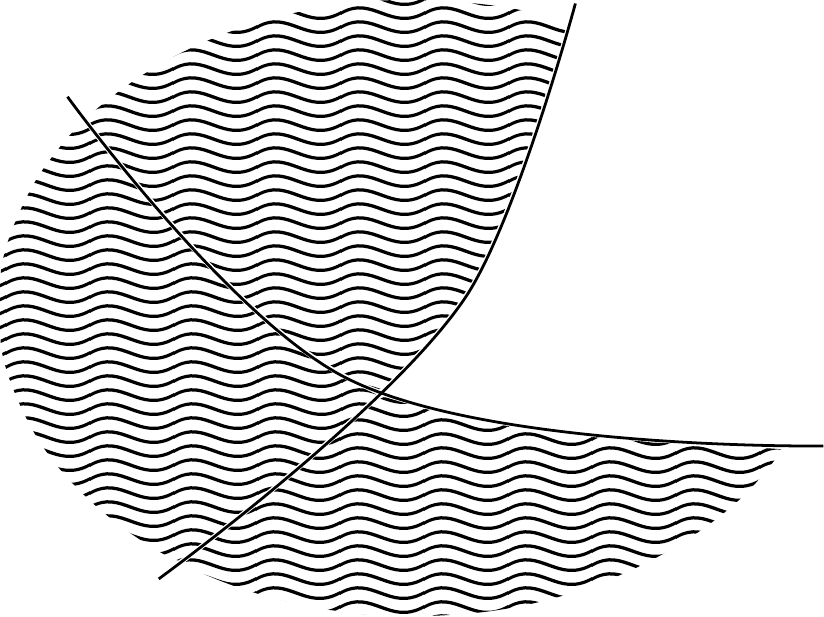
\caption{The forbidden configuration in a closed set.  A cellset is closed exactly when this configuration never occurs.}
\label{fig15-closed-forbidden}
\end{figure}
The closure $\overline{S}$ exists because the intersection of closed cellsets is closed.
\begin{lemma}\label{closure-facts}
A cellset $S$ is closed if and only if it has no simple extensions.  If $S'$ is a simple
extension of $S$, then $\overline{S'} = \overline{S}$.  Also, there is a chain of cellsets
\[ S = S_1 \subset S_2 \subset \cdots \subset S_n = \overline{S} \]
where $S_{i+1}$ is a simple extension of $S_i$ for each $i$.  Moreover $\rank(\overline{S}) \le \rank(S)$,
with equality if and only if $S_{i+1}$ is a safe simple extension of $S_i$ for each $i$.
\end{lemma}
\begin{proof}
The first claim is obvious.  The second claim follows because $S \subset S'$ implies that
$\overline{S} \subseteq \overline{S'}$, but conversely $S'$ is clearly a subset
of $\overline{S}$ and so $\overline{S'} \subseteq \overline{S}$ because $\overline{S'}$ is the smallest
closed subset containing $\overline{S}$.

For the third claim, inductively define $S_{i+1}$ to be any simple
extension of $S_i$, stopping at $S_n$ which has no simple extension.  Then $S_n$ is closed by the first
claim, and $S_n = \overline{S_n} = \overline{S}$ by the second claim.

For the final claim, we know by the previous lemma that $\rank(S_{i+1}) \le \rank(S_i)$ for every $i$.
Then by transitivity $\rank(\overline{S}) \le \rank(S)$, and if equality holds then
$\rank(S_{i+1}) = \rank(S_i)$ for every $i$, so that each step in the chain is a safe simple extension.
\end{proof}

The significance of rank and closure lies in the following:
\begin{theorem}\label{safe-sets-work}
Let $S$ be a cellset and $\phi$ be a labelling of $S$.  If $S'$ is a simple extension of $S$,
then $\phi$ extends to a labelling of $S'$ in at most one way, and if $S'$ is a safe simple extension
of $S$, then $\phi$ extends in exactly one way.  Consequently, $\phi$ extends to a labelling of $\overline{S}$
in at most one way, and if $\rank(\overline{S}) = \rank(S)$, then it extends in exactly one way.
\end{theorem}
\begin{proof}
First note that Equation (\ref{consistency-equation}) above has the property that any three of $\phi(w), \phi(x), \phi(y)$, and $\phi(z)$ uniquely
determine the fourth, because we stipulated that $\gamma_e$ be bijective.

If $S' = S \cup \{x\}$, then by definition of simple extension, there is some interior vertex $v$ of the medial graph such that
$x$ is one of the four cells around $v$, and the other three are already in $S$.  Then $\phi(x)$ is completely determined by (\ref{consistency-equation})
at the edge corresponding to $v$.  Moreover, if $S \cup \{x\}$ is a safe simple extension, then $v$ is the unique vertex which determines $\phi(x)$,
so $\phi(x)$ is uniquely determined.  Thus a labelling of $S$ extends to a labelling of $S'$ in at most one way, and at least one way if $S'$
is a safe simple extension of $S$.

The statements about extending $\phi$ to $\overline{S}$ then follow easily from the previous lemma.
\end{proof}

\begin{definition}
A cellset $S$ is \emph{safe} if $\rank(\overline{S}) = \rank(S)$.
\end{definition}
We have just shown that if $S$ is safe, then a labelling of $S$ extends to a labelling of $\overline{S}$ in a unique way.

\section{Convexity in Medial Graphs}\label{sec:convexity}
From now on we will assume that our medial graph $M$ is critical, so that each geodesic starts and ends on the boundary,
no geodesic intersects itself, and no two geodesics intersect more than once.  The results of this section are completely
independent of conductances, and are entirely combinatorial in nature.  Since critical medial graphs are equivalent to simple pseudoline arrangements,
our results can all be viewed as statements about pseudoline arrangements, but we will prefer to work with medial graphs (which contain a boundary curve)
to simplify some proofs. 

If $g$ is a geodesic in a medial graph, the Jordan curve theorem implies that $g$ divides the medial graph into two parts, which we call
\emph{pseudo halfplanes}.  We can identify these with cellsets.
\begin{definition}
A cellset $S$ is \emph{convex} if it is an intersection of zero or more pseudo halfplanes.
The \emph{convex closure} of a cellset $S$ is the smallest convex set containing
$S$, i.e., the intersection of all half-planes containing $S$.
\end{definition}
Since half-planes are closed, so are all convex sets.  We will roughly show the converse.  In doing this, we will show that our notion
of convexity has many of the usual properties of convexity.  Many of our results would be obvious if our medial graphs were always equivalent to medial graphs
drawn with straight lines, but this is not always the case.\footnote{On page 101 of \cite{Ringel}, Ringel gives a simple pseudoline arrangement which cannot be drawn
with straight lines.  It is obtained by perturbing the Pappus configuration.}

\begin{definition}
If $a$ and $b$ are two cells in a medial graph, we say that $a$ and $b$ are \emph{adjacent} if they share an edge.  A \emph{path}
between $a$ and $b$ is a sequence of cells $a = x_0, x_1, \ldots, x_n = b$ for $n \ge 0$, with $x_i$ adjacent to $x_{i+1}$ for each $i$.
The number $n$ is the \emph{length} of the path.  A \emph{minimal path} between two cells is a path of minimal length.  The \emph{distance between $a$
and $b$}, denoted $\dist(a,b)$, is defined to be the length of a minimal path between $a$ and $b$.
A cellset $S$ is \emph{connected}
if for every $a$ and $b$ in $S$, there is a path between $a$ and $b$ containing only cells in $S$. The \emph{connected components} of $S$
are the maximal connected subsets of $S$.
\end{definition}
Each cellset is the disjoint union of its connected components.  The medial graph as a whole is always connected, so that $\dist(a,b)$ is well-defined.

\begin{lemma}\label{connectedness-preserved}
If $S$ is a connected cellset, then so is $\overline{S}$.
\end{lemma}
\begin{proof}
Because $\overline{S}$ can be obtained from $S$ by a series of simple extensions (by Lemma~\ref{closure-facts}), it suffices to show
that a simple extension of a connected cellset is connected, which is obvious.
\end{proof}

\begin{definition}
Let $g$ be a geodesic, and $a$, $b$ be cells.  We say that $g$ \emph{separates} $a$ and $b$ if $a$ and $b$ are on opposite sides of $g$.
We let $\dist'(a,b)$ denote the number of geodesics which separate $a$ and $b$.  If
$S$ is a cellset, we say that $g$ \emph{separates cells of $S$} if there exist $a, b \in S$ on opposite sides of $g$.
\end{definition}
Clearly, $\dist'(a,b) \le \dist(a,b)$.

\begin{theorem}\label{convex-connected}
Let $S$ be a nonempty convex cellset.  Then there is a critical medial graph $M'$ whose cells are the same as the cells in $S$.
In particular, $S$ is connected
and the boundary of the closure of the union of the cells in $S$ is a Jordan curve.  
Additionally, the geodesics in $M'$ are in one to one correspondence with the geodesics of $M$ which separate cells of $S$.
\end{theorem}
\begin{proof}
If $H$ is a pseudo halfplane, then we can make a medial graph out of the cells in $H$, as shown in Figure~\ref{fig16-halfplane-and-medial-graph}.
This new medial graph will still be critical.
Moreover, if $H'$ is any other
pseudo halfplane, then $H' \cap H$ is a pseudo halfplane in the new medial graph, by criticality.  We show by induction on $r$ that
if $M$ is a critical medial graph and $S$ is the intersection of $r$ pseudo halfplanes in $M$, then $S$ is the set of cells of a critical medial graph.
The base case where $r = 0$ is trivial.  Otherwise, write $S$ as $H \cap S'$, where $H$ is a half-plane, and $S' = \bigcap_{i = 1}^{r-1} H_i$ for pseudo halfplanes
$H_i$.  Then we can consider $H$ as a medial graph $M'$, and $H \cap H_i$ is a pseudo halfplane of $M'$ for each $i$.  By induction,
$S = H \cap S' = \bigcap_{i = 1}^{r-1} (H \cap H_i)$ is the set of cells of a critical medial graph, and we are done.
%
We leave the second claim as an exercise to the reader.
\end{proof}
\begin{figure}
\centering
\def \svgwidth{5in}
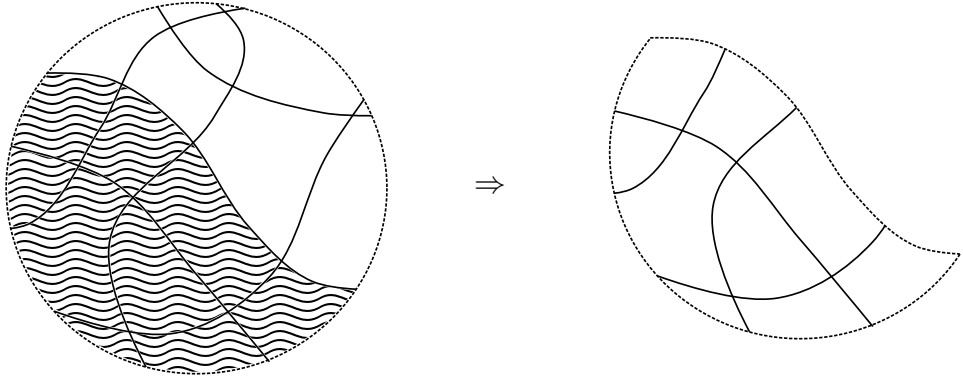
\caption{A pseudo halfplane (left), considered as a medial graph in its own right (right).}
\label{fig16-halfplane-and-medial-graph}
\end{figure}

\begin{corollary}\label{distance-by-geodesics}
If $a$ and $b$ are two cells of a critical medial graph $M$, then $\dist'(a,b) = \dist(a,b)$.
If $c$ is a third cell, then $c$ is on a minimal path between $a$ and $b$ if and only if $c$ is in the convex closure of $\{a,b\}$.
\end{corollary}
\begin{proof}
We proceed by induction on $\dist'(a,b)$.
If $a$ and $b$ are the same cell, then clearly $\dist(a,b) = \dist'(a,b) = 0$.  Assume therefore that $a \ne b$.
Let $S$ be the convex closure of $\{a,b\}$.
By the theorem, $S$ is connected.
Since $S$ is connected, there is a path from $a$ to $b$ in $S$, and in particular there is some cell $a' \in S$ which
is adjacent to $a$.  Let $g$ be the geodesic which separates $a'$ from $a$.  Then $g$ also separates $a$ from $b$. Otherwise, the pseudo halfplane
$H$ on the side of $g$ containing $a$ and $b$ would not contain $a'$, so $a' \notin S$.  Since $g$ separates $a$ from $b$ and $a$ from $a'$,
it follows that $\dist'(a',b) = \dist'(a,b) - 1$.  By induction, $\dist'(a,b) = \dist'(a',b) + 1 = \dist(a',b) + 1$.  Therefore,
\[ \dist'(a,b) \le \dist(a,b) \le \dist(a,a') + \dist(a',b) = \dist(a',b) + 1 = \dist'(a,b),\]
so the desired equality holds.

For the second claim, note that $c$ is on a minimal path between $a$ and $b$ if and only if $\dist(a,b) = \dist(a,c) + \dist(c,b)$.
Define $\sigma_g(x,y)$ for cells $x, y$ and geodesic $g$ to be 1 if $g$ separates $a$ from $b$, and 0 otherwise.  Then the
first claim of this corollary implies that $\dist(x,y) = \sum_g \sigma_g(x,y)$.  Now any geodesic which separates $a$ from $b$ must separate
$a$ from $c$ or $c$ from $b$, so
\[ \sigma_g(a,b) \le \sigma_g(a,c) + \sigma_g(c,b).\]
Summing over $g$, we see that $\dist(a,b) \le \dist(a,c) + \dist(c,b)$, with equality if and only if $\sigma_g(a,b) = \sigma_g(a,c) + \sigma_g(c,b)$
for every geodesic $g$.  Thus $c$ is on a minimal path from $a$ to $b$ if and only if for every geodesic $g$,
either $g$ does not separate any of $a,b,c$, or $g$ separates $a$ from $b$.  Equivalently, $c$ is on a minimal path from $a$ to $b$ if and
only if $c$ is in every half plane which contains both $a$ and $c$.  But this is just the convex closure of $\{a,b\}$.
\end{proof}

\begin{corollary}\label{cells-are-nice}
Every cell in a critical medial graph has a boundary that is a Jordan curve.
\end{corollary}
\begin{proof}
Let $a$ be a cell, and let $S$ be the convex closure of $\{a\}$.  If $b$ is any adjacent cell to $a$,
then some geodesic $g$ separates $a$ from $b$.  Then there is a pseudo halfplane $H$ which contains $b$ but not
$a$, so that $b \notin S$.  Thus $S$ contains no cells adjacent to $a$.  But by Theorem~\ref{convex-connected}, $S$ is connected, so
$S = \{a\}$.  It then follows that $a$ itself must have the shape of a medial graph, and in particular its boundary must be a piecewise
smooth Jordan curve.
\end{proof}

\begin{definition}
Let $(x_1,x_2,\ldots,x_n)$ and $(y_1,y_2,\ldots,y_n)$ be two paths from $a$ to $b$.  We say that $(y_1,\ldots,y_n)$ is a \emph{simple deformation}
of $(x_1,\ldots,x_n)$ if there is some $1 < i < n$ such that $x_j = y_j$ for $j \ne i$, and $x_{i-1}, x_{i+1}, x_i$ and $y_i$ are the four cells
surrounding some internal vertex.  We say that $(y_1,y_2,\ldots,y_n)$ is a \emph{deformation} of $(x_1,\ldots,x_n)$ if it is obtained
by a series of zero or more simple deformations.
\end{definition}

\begin{theorem}\label{path-homotopy}
If $a, b$ are cells, then all minimal paths between $a$ and $b$ are deformations of each other.
\end{theorem}
\begin{proof}
We proceed by induction on $\dist(a,b)$.  The base cases where $\dist(a,b) = 0$ or 1 are obvious, so assume that $\dist(a,b) > 1$.

By Corollary~\ref{distance-by-geodesics}, the union of all the minimal paths between $a$ and $b$ are the convex closure of $\{a,b\}$,
which by Theorem~\ref{convex-connected} is a medial graph itself.  So we can assume without loss of generality that the convex
closure of $\{a,b\}$ is everything, and so every geodesic separates $a$ from $b$.

Let $S$ be the set of all cells immediately adjacent to $a$.  We claim that we can enumerate the elements of $S$ as $s_1, s_2, \ldots, s_n$
in such a way that for each $i$, the three cells $a$, $s_i$, and $s_{i+1}$ meet at a vertex (with a fourth cell).  This setup is shown in Figure~\ref{fig18-contiguity}.
This is obvious
if $a$ is not a boundary cell (which probably never happens).  If $a$ is a boundary cell, then this is still clear, unless
$a$ meets the boundary of the medial graph in two or more disjoint boundary arcs, as in Figure~\ref{fig18-point-5-bad-cell}.
But in this case, we can add a new geodesic connecting two of
these boundary arcs, and throw away the side of the geodesic that is opposite $g$, making a smaller medial graph.  Any geodesic discarded by this operation
could not have separated $a$ from $b$, and could not have been present.  Consequently, the discarded part could not have contained any additional cells or
geodesics, implying that the two boundary arcs we connected were in fact the same, a contradiction.
\begin{figure}
\centering
\def \svgwidth{4in}
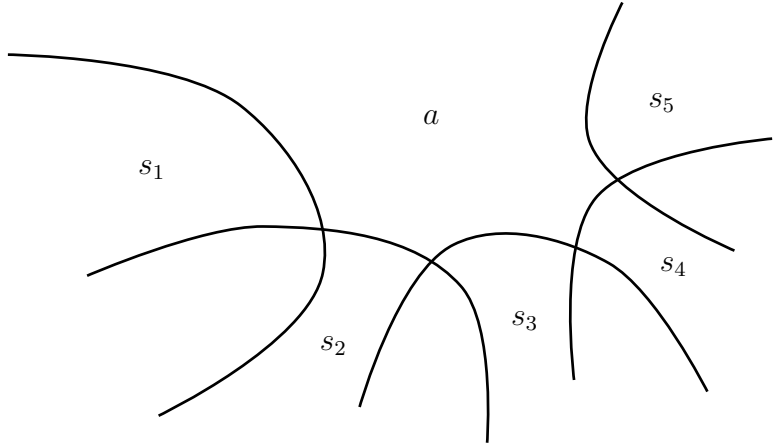
\caption{The desired setup in the proof of Theorem~\ref{path-homotopy}.  Here, we can enumerate the neighbors of $a$ in order.  In particular, they occur
contiguously.}
\label{fig18-contiguity}
\end{figure}
\begin{figure}
\centering
\def \svgwidth{2.5in}
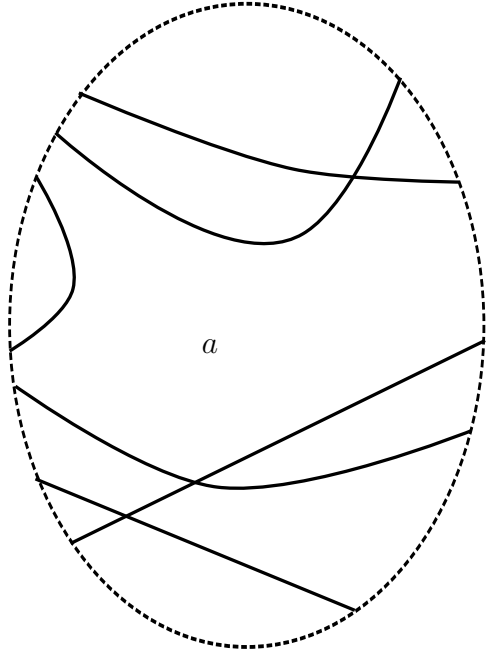
\caption{A cell whose neighbors are not contiguous.  This only occurs if the cell touches the boundary more than once, as $a$ does here.}
\label{fig18-point-5-bad-cell}
\end{figure}

So we can number the neighbors of $a$ as $s_1, s_2, \ldots, s_n$.  Any minimal path from $a$ to $b$ must begin by moving from
$a$ to one of the $s_i$.  By induction, we know that all minimal paths from $s_i$ to $b$ are equivalent through deformations.  So it suffices
to show for each $i$ that at least one path from $a$ to $b$ through $s_i$ is a deformation of a path from $a$ to $b$ through $s_{i+1}$.

To see this, let $c$ be the fourth cell meeting $a$, $s_i$, and $s_{i+1}$, as in Figure~\ref{fig19-minimals-homotopic}.
If we let $g_i$ and $g_{i+1}$ be the geodesics separating
$a$ from $s_i$ and $s_{i+1}$, respectively, then both $g_i$ and $g_{i+1}$ separate $a$ from $b$.  Consequently, $\dist(c,b) = \dist(a,b) - 2$.
Thus if we take a minimal path from $c$ to $b$, and add on $(a,s_i)$ or $(a,s_{i+1})$ to the beginning, we will get minimal paths from $a$ to $b$.
But these two alternatives are deformations of each other, so we are done.
\begin{figure}
\centering
\def \svgwidth{3.5in}
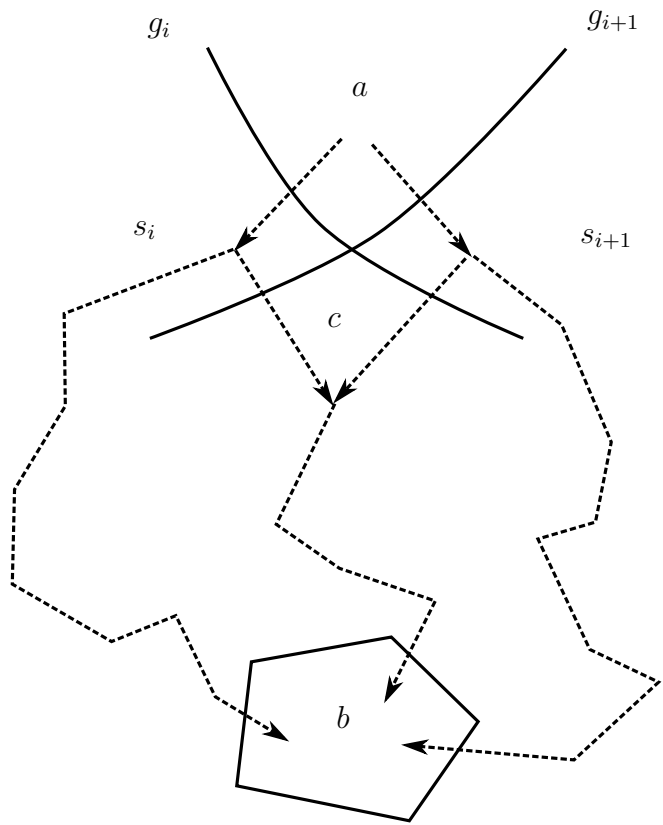
\caption{The setup of the rest of the proof of Theorem~\ref{path-homotopy}.}
\label{fig19-minimals-homotopic}
\end{figure}
\end{proof}

The following theorem gives several conditions equivalent to convexity.  
\begin{theorem}\label{tfae}
Let $S$ be a cellset.  The following are equivalent:
\begin{description}
\item[(a)] $S$ is closed and connected.
\item[(b)] $S$ is connected, and for every $a, b \in S$, if one minimal path between $a$ and $b$ is in $S$, then every minimal path between $a$ and $b$ is in $S$.
\item[(c)] If $a, b \in S$, then every minimal path from $a$ to $b$ is in $S$.
\item[(d)] If $a, b$ are adjacent cells separated by geodesic $g$, and $a \in S$ but $b \notin S$, then every cell in $S$
is on the same side of $g$ as $a$.
\item[(e)] $S$ is convex
\end{description}
\end{theorem}
\begin{proof}
~
\begin{description}
\item[(a)$\Rightarrow$(b)]
By Theorem~\ref{path-homotopy}, it suffices to show that if $S$ is closed, and $P$ is a path in $S$, then every deformation of $P$
is in $S$.  This is trivial, by definition of closed.
\item[(b)$\Rightarrow$(c)]
Let $S$ satisfy (b).  We first show that $S$ is closed.  Let $w, x, y, z$ be four cells in $S$ that meet around an interior vertex, in counterclockwise order,
Suppose that $x,y,z \in S$.  If $x = z$, then the cell $x$ contradicts Corollary~\ref{cells-are-nice}.  So $x \ne z$.  And since medial graphs are always
two-colorable, $x$ and $z$ are not adjacent.  Consequently $x \to w \to z$ and $x \to y \to z$ are both minimal paths from $x$ to $z$.  Since the former is in $S$,
the latter must also be in $S$, showing that $S$ is closed.

Using this, we show that for every $a, b \in S$, there is a minimal path between $a$ and $b$ in $S$, which suffices to show (c) given (b).  Since $S$
is connected, between any two $a, b \in S$ there is at least one path.  Call a path \emph{$S$-minimal} if it has minimal length among all paths that lie in $S$.
Then it remains to show that every $S$-minimal path is truly minimal, since an $S$-minimal path always exists between two cells in $S$.

We proceed by induction on the length $n$ of our $S$-minimal path.  The base cases where $n < 2$ are trivial.  So suppose that $x_0 \to x_1 \to  \cdots \to x_n$
is an $S$-minimal path for $n \ge 2$.  Then the path $x_1 \to \cdots \to x_n$ is also $S$-minimal, so by induction it
is truly minimal.  In particular,
\[ \dist(x_1,x_n) = n - 1.\]
Now let $g$ be the geodesic which separates $x_0$ from $x_1$.  If $x_n$ is on the same side of $g$ as $x_1$, then by Corollary~\ref{distance-by-geodesics},
$\dist(x_0,x_n) = \dist(x_1,x_n) + 1 = n$ and we are done.  So we can assume that $x_n$ is on the same side of $g$ as $x_0$.

Similarly, the path $x_0 \to \cdots \to x_{n-1}$ is minimal.  Since it begins with a step across the geodesic $g$, $x_{n-1}$ must be on the opposite
side of $g$ from $x_0$.  This makes $x_{n-1}$ and $x_n$ on opposite sides of $g$, so they must be separated
by $g$.  In other words, the path $x_0 \to \cdots \to x_n$ begins and ends by stepping across $g$.  Now the path $x_1 \to \cdots \to x_{n-1}$ is $S$-minimal, so by
induction it is truly minimal. But then property (b) implies that every minimal path between $x_1$ and $x_{n-1}$ is in $S$.  In particular,
the path from $x_1$ to $x_{n-1}$ that proceeds directly along the side of geodesic $g$, as in Figure~\ref{fig20-proof-idea}, must be in $S$.  This path is minimal
because it steps over no geodesic twice, by criticality.
\begin{figure}
\centering
\def \svgwidth{5in}
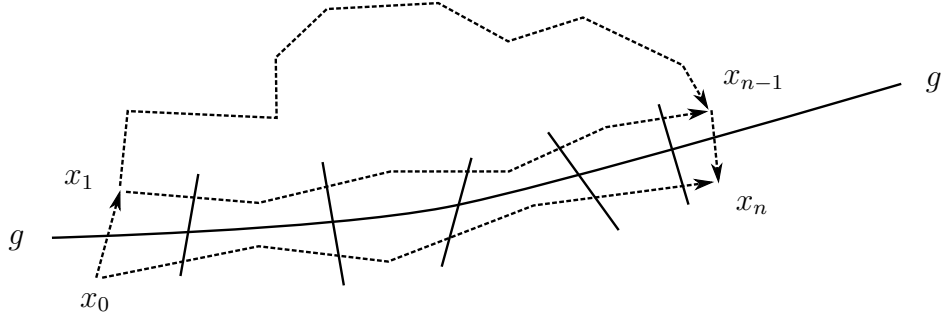
\caption{The path in question must begin and end by crossing $g$.  By induction, the path that proceeds directly from $x_1$ to $x_{n-1}$ along the side of $g$
must be in $S$.  Once all these cells are in $S$, the cells on the near side of $g$ must also be in $S$, because $S$ is closed.  This shows that a true minimal
path from $x_0$ to $x_n$ is in $S$.}
\label{fig20-proof-idea}
\end{figure}

Consequently, every cell on the \emph{far} side of $g$ from $x_1$ to $x_{n-1}$ is in $S$, as are the cells $x_0$ and $x_n$ which are opposite
$x_1$ and $x_{n-1}$, respectively.  Thus we have the setup of Figure~\ref{fig20-proof-idea}.
But then the fact that $S$ is closed easily implies that every cell on the \emph{near} side of $g$
from $x_0$ to $x_n$ is in $S$.  In particular, there is a truly minimal path from $x_0$ to $x_n$ in $S$, so every $S$-minimal path from $x_0$ to $x_n$
must be truly minimal.
\item[(c)$\Rightarrow$(d)]
Let $S$ satisfy (c), and suppose that $a$ and $b$ are adjacent cells, with $a \in S$ and $b \notin S$.  Let $H$ be the pseudo-halfplane
containing $a$ but not $b$.
We wish to show that $S \subseteq H$.  Suppose for the sake of contradiction that $c \in S$ and $c \notin H$.
Then by property (c) and Corollary~\ref{distance-by-geodesics}, the convex closure of $\{a,c\}$ is in $S$.  We claim that $b \in \overline{\{a,c\}} \subseteq S$,
a contradiction.  Otherwise, there must be some pseudo-halfplane $H'$ containing $a$ and $c$ but not $b$.  But then the geodesic which defines $H'$ must
separate $a$ from $b$, so $H'$ must be $H$ or its complement.  Either way, $H'$ cannot contain both $a$ and $c$, since
$a \in H$ and $c \notin H$.
\item[(d)$\Rightarrow$(e)]
Let $S$ satisfy (d).  If $S$ is empty, the result is trivial, so suppose that $S$ is nonempty.
Call $g$ a \emph{defining geodesic} of $S$ if there are adjacent cells $a,b$ separated by $g$, with $a \in S$ and $b \notin S$.
In such a circumstance, call the half-plane on the side of $g$ containing $a$ a \emph{defining halfplane}.
Then condition (d) implies that $S$ is contained in every defining halfplane.  Let $S'$ be the intersection of all defining halfplanes.  We
have $S \subseteq S'$, and $S'$ is convex.  It remains to show that $S = S'$.  Otherwise there exists some cell in $S' \setminus S$.
By Theorem~\ref{convex-connected}, $S'$ is connected, so there is some path from a cell in $S$ to a cell in $S' \setminus S$.  In particular,
there exists some pair of adjacent cells $a,b \in S'$ with $a \in S$ and $b \notin S$.  But then the halfplane containing $a$ but not $b$ is a
defining halfplane of $S$, so $b \notin S'$, a contradiction.
\item[(e)$\Rightarrow$(a)]
Let $S$ be convex.  Then Theorem~\ref{convex-connected} shows that $S$ is connected.  Convex sets are always closed, because half-planes are closed
and intersections of closed sets are closed.
\end{description}
\end{proof}

\begin{corollary}\label{characterize-closedness}
Let $S$ be a cellset.  Then $S$ is closed if and only if each connected component of $S$ is convex.
\end{corollary}
\begin{proof}
It is easy to check that $S$ is closed if and only if each connected component of $S$ is closed.  The result
then follows by the equivalence (a)$\iff$(e) of Theorem~\ref{tfae}.
\end{proof}

Next we return to the notion of rank:
\begin{theorem}\label{rank-formula}
If $S$ is a convex cellset, then $\rank(S)$ is one plus the number of geodesics $g$ which separate cells
of $S$.
\end{theorem}
\begin{proof}
By Theorem~\ref{convex-connected}, it suffices to show this in the case where $S$ is the entire medial graph.
Then we need to show that the rank of $S$ is one plus the number of geodesics.  Let $c$ count the number of cells,
$v_i$ count the number of interior vertices, $v_b$ count the number of boundary vertices, and $e$ count the number
of edges including the segments of the boundary curve.  Interior vertices have degree 4 while boundary vertices have degree 3,
so that
\[ 2e = 4v_i + 3v_b.\]
Since all cells in a critical medial graph are homeomorphic to disks (by Corollary~\ref{cells-are-nice}), Euler's formula applies, yielding
\[ c + v_i + v_b = e + 1 = 2v_i + \frac{3}{2}v_b + 1.\]
Thus
\[ \rank(S) = c - v_i = \frac{3}{2}v_b - v_b + 1 = \frac{v_b}{2} + 1.\]
But $v_b/2$ counts the number of geodesics, because each geodesic has two endpoints.
\end{proof}

Recall that a cellset is ``safe'' if its closure has the same rank as it.
The main statement that we will need to set up our boundary value problems is the following:
\begin{lemma}\label{extend-by-one}
Let $S$ be a safe cellset with connected closure.  Let $b \notin \overline{S}$ be a cell adjacent to a cell
in $\overline{S}$.  Then $S \cup b$ is a safe cellset with connected closure.
\end{lemma}
\begin{proof}
Note that $\overline{S \cup \{b\}} = \overline{\overline{S} \cup \{b\}}$ by general abstract properties of closure operations.
Since $b$ is adjacent to $\overline{S}$ and $\overline{S}$ is connected, Lemma~\ref{connectedness-preserved} shows
that $\overline{S \cup \{b\}} = \overline{\overline{S} \cup \{b\}}$ is connected.

Since both $\overline{S}$ and $\overline{S \cup \{b\}}$ are closed and connected, they are convex by Theorem~\ref{tfae}.  So
Theorem~\ref{rank-formula} tells us that the ranks of $\overline{S}$ and $\overline{S \cup \{b\}}$ are determined by the number
of geodesics which cut across them.  Now any geodesic which cuts across $\overline{S}$ will also cut across $\overline{S \cup \{b\}}$,
so
\[ \rank(\overline{S}) \le \rank(\overline{S \cup \{b\}}).\]
But in fact at least one geodesic which cuts across $\overline{S \cup \{b\}}$ does not cut across $\overline{S}$, specifically
the geodesic $g$ which separates $b$ from its neighbor in $\overline{S}$.  This geodesic cannot cut across
$\overline{S}$ by condition (d) in Theorem~\ref{tfae}.  So we have
\begin{equation} \rank(\overline{S}) + 1 \le \rank(\overline{S \cup \{b\}}).\label{t1}\end{equation}

Now
\[ \rank(S) = \rank(\overline{S})\]
because $S$ is safe.  And since $b \notin \overline{S}$, $S \cup \{b\}$ is not a simple extension of $S$,
so that
\[ \rank(S \cup \{b\}) = \rank(S) + 1\]
by Lemma~\ref{rank-and-extension}.  And by Lemma~\ref{closure-facts},
\[
\rank(\overline{S \cup \{b\}}) \le \rank(S \cup \{b\}) = \rank(S) + 1 = \rank(\overline{S}) + 1.\]
Combining this with (\ref{t1}), we see that equality must hold, so in particular
\[ \rank(\overline{S \cup \{b\}}) = \rank(S \cup \{b\}),\]
so $S \cup \{b\}$ is safe.
\end{proof}

We will use this lemma to inductively build up safe sets of boundary cells.

\section{Nonlinear Recovery}\label{sec:nonlinear-recovery}
To complete the proof of Theorem~\ref{main-result}, we need to show how to recover the conductance functions of boundary triangles
in critical medial graphs.%
\begin{lemma}\label{problem-set-up}
Let $M$ be a critical medial graph, $b$ be a boundary cell, and $g$ be one of the geodesics by $b$, as in Figure~\ref{fig24-theorem-setup}.
Then there are two sets of boundary cells $T, S$
such that
\begin{itemize}
\item $S$ is a safe cellset: $\rank(S) = \rank(\overline{S})$.
\item $T \subseteq S$.
\item $b \in S \setminus T$.
\item $\overline{T}$ is the halfplane on the far side of $g$ from $b$.
\item $\overline{S}$ is the entire medial graph.
\end{itemize}
\end{lemma}
\begin{figure}
\centering
\def \svgwidth{4in}
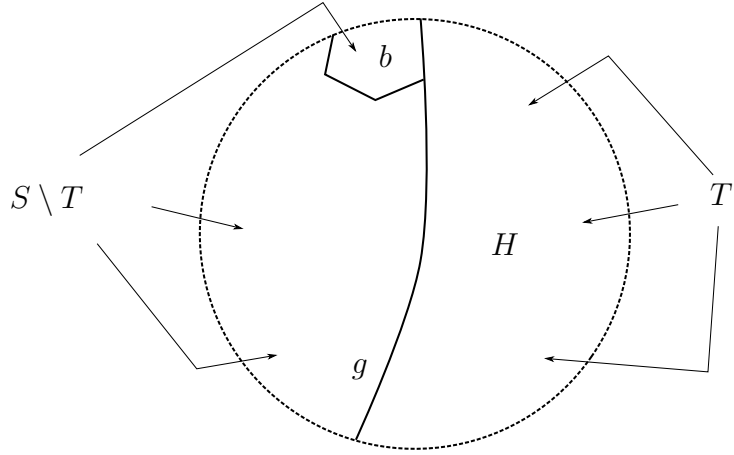
\caption{The setup of Lemma~\ref{problem-set-up}.  Given boundary cell $b$ and adjacent geodesic $g$, we can find sets $S$ and $T$ of boundary cells
having the desired properties.}
\label{fig24-theorem-setup}
\end{figure}
\begin{proof}
Let $H$ be the half-plane on the far side of $g$ from $b$.  Number the boundary cells in $H$
$a_1, a_2, \ldots, a_n$ in consecutive order.  Let $T_0 = \{a_1\}$, and then recursively define
$T_{i+1}$ to be $T_i \cup \{a_{j_i}\}$ where $j_i$ is the minimal $j$ such that $a_j \notin \overline{T_i}$.  Stop
once $\overline{T_i}$ contains every $a_j$.  This gives us a chain of sets of boundary cells
\[ T_1 \subset T_2 \subset \cdots \subset T_n.\]
By choice of $a_{j_i}$, we must have $a_{j_i - 1} \in \overline{T_i}$ for each $i$.  Consequently,
$a_{j_i}$ is adjacent to a cell in $\overline{T_i}$.  Then by Lemma~\ref{extend-by-one}, we see that
if $T_i$ is safe with connected closure, so is $T_{i+1} = T_i \cup \{a_{j_i}\}$.  Since $T_1 = \{a\}$ clearly
has $\overline{T_1} = T_1$, $T_1$ is safe with connected closure.  Thus every $T_i$ is safe with connected closure.

Let $T = T_n$.  Then $T$ is safe with connected closure, and $\overline{T}$ contains every boundary
cell in $H$.  Since $T \subseteq H$, $\overline{T} \subseteq \overline{H} = H$.  But conversely, if any half-plane $H'$
contains $\overline{T}$, then its defining geodesic $g'$ must start and end on the far side
of $H'$, so by criticality it cannot cross $g$, and $H \subseteq H'$.  Consequently the convex closure of $\overline{T}$
contains $H$.  But by Theorem~\ref{tfae}, $\overline{T}$ is already convex, so we see that $H \subseteq \overline{T}$.
So $H = \overline{T}$.

We then let $T_{n+1} = T_n \cup \{b\}$.  Now $b \notin \overline{T} = H$, so by Lemma~\ref{extend-by-one} again, $T_{n+1}$
is still safe with connected closure.  Continuing on inductively we can keep adding more boundary cells until we have
a set $S$ which is safe, has connected closure, contains $T \cup \{b\}$, and has every boundary cell in its closure.
But $\overline{S}$ is convex by Theorem~\ref{tfae}, and any convex set containing every boundary cell must be the whole medial graph,
since no halfplane contains every boundary cell.

See Figure~\ref{fig24-theorem-illustration} for an illustration of all these steps in a small example.
\end{proof}
\begin{figure}
\centering
\def \svgwidth{5in}
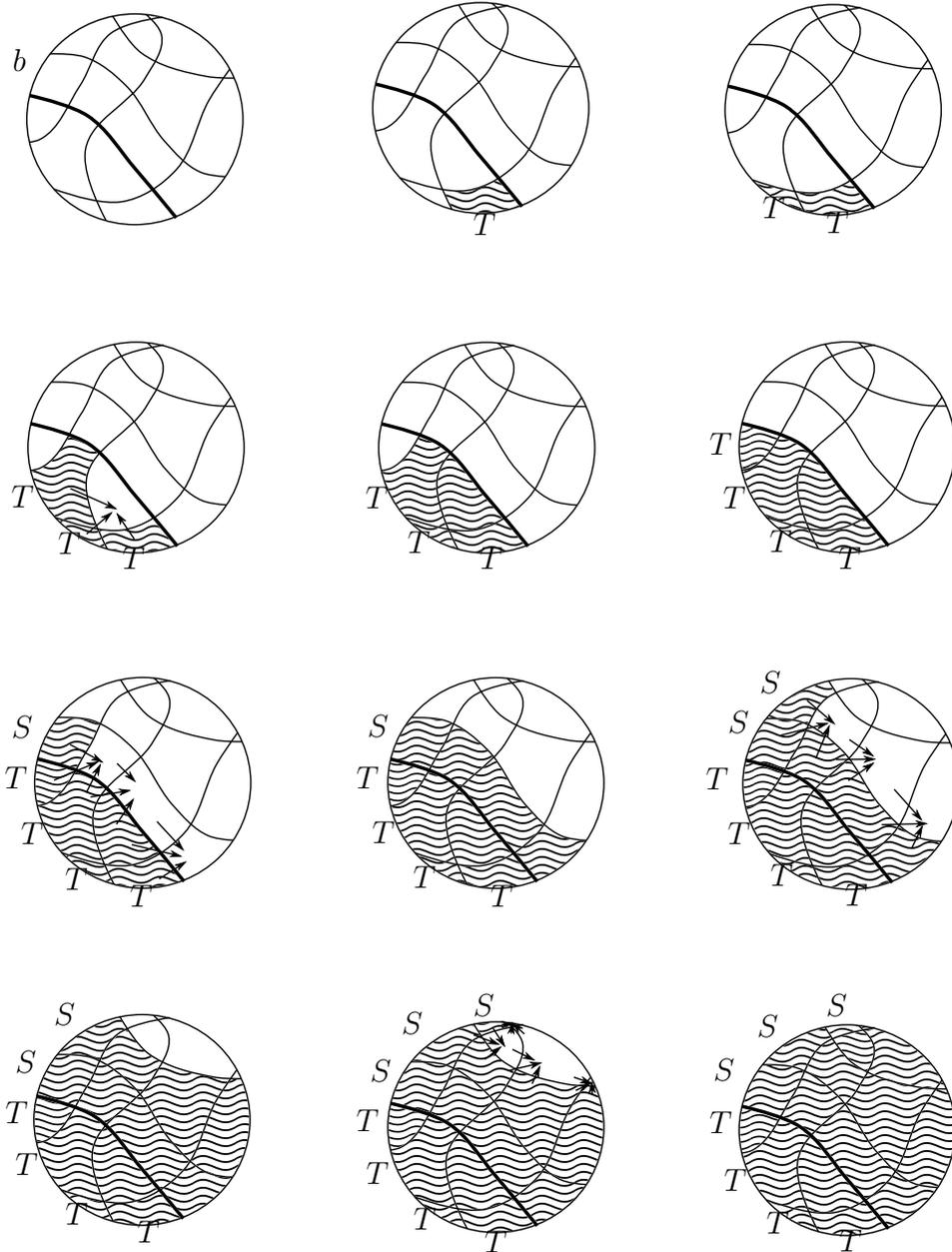
\caption{A small example of the proof of Lemma~\ref{problem-set-up}.  By a series of safe simple extensions, we first fill up the desired pseudo halfplane (the one
below the darkened geodesic), and then proceed to fill up the rest of the medial graph.  Arrows show the propagation of information.  Here, $g$ is the darkened geodesic
and the boundary cell $b$ is the one above the left end of $g$.  $S$ consists of all four boundary cells below $g$, while $T$ additionally contains $b$
and the next two boundary cells after $b$ in clockwise order.}
\label{fig24-theorem-illustration}
\end{figure}

Using this we prove the other direction of Theorem~\ref{main-result}
\begin{proof}[Proof (of Theorem~\ref{main-result})]
By Theorem~\ref{conditional}, we only need to show that conductance functions of boundary triangles are recoverable.
%
%
Let $b$ be a boundary triangle, and let $g$ be one of the two geodesics which define $b$.
Find cellsets $T$ and $S$ as in Lemma~\ref{problem-set-up}.  Fix boundary values of $x$ at cell $b$, and $0$ on $S \setminus \{b\}$.
Since $S$ is safe and has closure the entire medial graph, it follows by Theorem~\ref{safe-sets-work} that these boundary values
uniquely determine a labelling $\phi$ of the entire medial graph.  The restriction of $\phi$ to $\overline{T}$ is a labelling of $\overline{T}$,
and so is the constant zero function.  Since these extend the same labelling of $T$, Theorem~\ref{safe-sets-work} shows that $\phi$ must identically
vanish on $\overline{T}$.  By Lemma~\ref{problem-set-up}, $\overline{T}$ contains a neighbor $a$ of $b$, and the cell $c$ that diagonally
touches $a$.  Letting $d$ be the fourth cell that meets at the apex of $b$, we have from Equation (\ref{consistency-equation})
\[ \phi(d) = \phi(d) - \phi(b) = \pm \gamma^{\pm}(\phi(a) - \phi(c)) = \pm \gamma^{\pm}(x) \]
where the signs depend on the orientation of $e$, the order of $a,b,c,d$ around their meeting point, and whether or not $e$ is a boundary spike.
But we can determine $\phi(d)$ from the boundary relation $\Xi$, since $d$ is a boundary cell.  It follows that we can read off $\pm \gamma^{\pm}(x)$
directly from $\Xi$.
%
%
\end{proof}

\section{Applications of Negative conductances}\label{sec:applications}
Even though our original motivation for Theorem~\ref{main-result} was the recovery of networks that had nonlinear
\emph{monotone} conductance functions, monotonicity was never needed in the proof.  In fact,
we could have allowed our currents and voltages to take values in an arbitrary abelian group.  This result is rather strong, and one wonders what sort
of applications it might have.

One interesting implication is that we can allow linear conductance functions $\gamma(x) = cx$ arbitrary $c \in \mathbb{C} \setminus \{0\}$,
and recovery will be possible.  While these
nonpositive conductances might seem arbitrary, there are a couple of circumstances where they come up naturally.
\subsection{Removing a mild failure of circular planarity}

\begin{figure}
\centering
\def \svgwidth{4.5in}
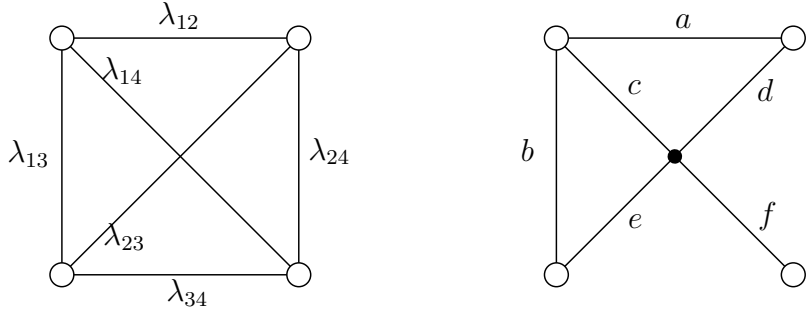
\caption{A complete graph on four vertices, and a critical circular planar graph.  Both have six edges, and generically, they are interchangeable,
after an appropriate birational change of variables.}
\label{fig26-well-connected}
\end{figure}
One circumstance is a generalization of the Y-$\Delta$ transformation of Section~\ref{sec:linear}.\footnote{The following idea appears in Konrad Schr{\o}der's
REU paper \cite{KonradSchroder} in 1995 and in Michael Goff's REU paper \cite{MichaelGoff} in 2003.
The latter paper claims to have a proof that signed conductances can
be recovered in critical circular planar networks, provided that the Dirichlet-Neumann relation is a function.  However, the proof seems flawed to me.}  Consider the
circular planar network shown on the right side of Figure~\ref{fig26-well-connected}, and suppose that all conductors are linear.  By a direct calculation,
one can show that this network is electrically equivalent to the network on the complete graph $K_4$ on the left side of Figure~\ref{fig26-well-connected}
with the conductances
specified as follows:
\begin{align*}
\lambda_{12} &= a + \frac{cd}{\sigma} \\
\lambda_{13} &= b + \frac{ce}{\sigma} \\
\lambda_{14} &= \frac{cf}{\sigma} \\
\lambda_{23} &= \frac{de}{\sigma} \\
\lambda_{24} &= \frac{df}{\sigma} \\
\lambda_{34} &= \frac{ef}{\sigma}
\end{align*}
where $\sigma = c + d + e + f$, provided that $c + d + e + f$.
Conversely, if the $\lambda_{ij}$ are positive numbers, then the $K_4$ on the left side of Figure~\ref{fig26-well-connected} is electrically equivalent
to the graph on the right side of Figure~\ref{fig26-well-connected} with conductances given as follows:
\begin{align*}
a &= \lambda_{12} - \frac{\lambda_{14}\lambda_{23}}{\lambda_{34}} \\
b &= \lambda_{13} - \frac{\lambda_{14}\lambda_{23}}{\lambda_{24}} \\
c &= \frac{q \lambda_{14}}{\lambda_{24} \lambda_{34}} \\
d &= \frac{q}{\lambda_{34}} \\
e &= \frac{q}{\lambda_{24}} \\
f &= \frac{q}{\lambda_{23}}
\end{align*}
where
\[ q = \lambda_{14} \lambda_{23} + \lambda_{23} \lambda_{24} + \lambda_{23} \lambda_{34} + \lambda_{24} \lambda_{34}\]

So we have a way of transforming
any $K_4$ (with positive conductances) into an electrically equivalent network on the graph on the right side of Figure~\ref{fig26-well-connected},
possibly introducing nonpositive conductances in the process.  Now from the equations for $a,\ldots,f$, we see
that only $a$ or $b$ can be zero or negative.  A conductance of zero is equivalent to a deleted edge, so
any $K_4$ (with positive conductances) can be turned into one of the four graphs of Figure~\ref{fig27-fourpossibilities}
(with conductances in $\mathbb{R} \setminus \{0\}$).
These graphs' associated medial graphs are also shown
in Figure~\ref{fig27-fourpossibilities}.
\begin{figure}
\centering
\def \svgwidth{4.5in}
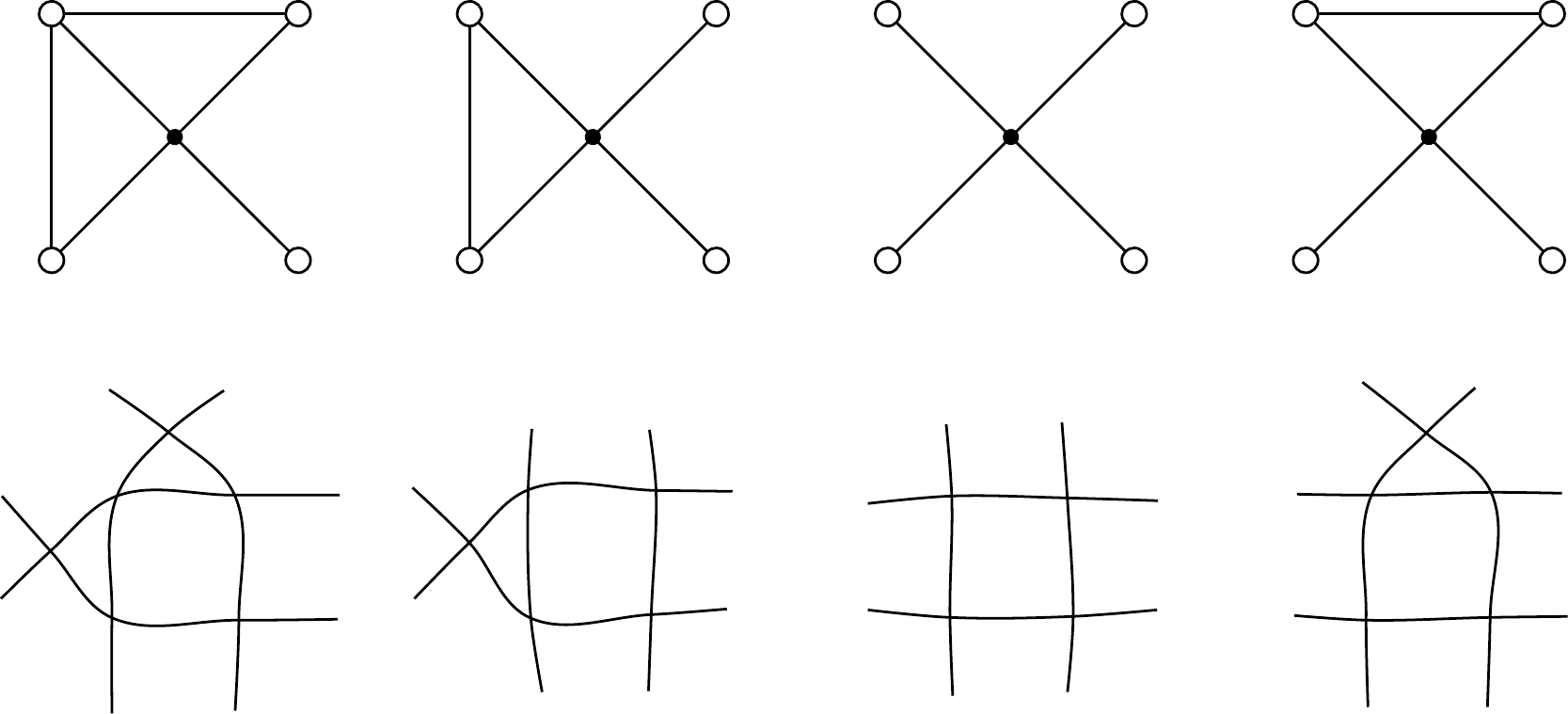
\caption{A $K_4$ with positive conductances
can be turned into one of the four critical circular planar graphs of the top row, possibly introducing negative conductances in the process.
The four associated medial graphs are shown on the bottom row.}
\label{fig27-fourpossibilities}
\end{figure}

Like a Y-$\Delta$ transform, we can apply this transformation to a $K_4$ embedded locally within a larger graph.
For example, the graph of Figure~\ref{example-k4} (with positive conductances) will be electrically equivalent to one of the graphs
of Figure~\ref{the-four-it-becomes} (with nonzero conductances).  Now these resultant graphs are circular planar, and it is easy to check
that they are critical.  Consequently, the original graph of Figure~\ref{example-k4}
is (weakly) recoverable - as long as the response matrix provides enough information to distinguish
the four graphs of Figure~\ref{the-four-it-becomes}.  This is implied by the following lemma, which generalizes
results in \cite{CandM}.
\begin{figure}
\centering
\def \svgwidth{2in}
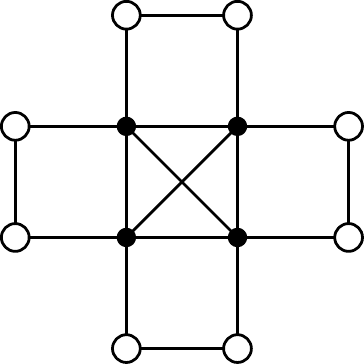
\caption{If all the conductances in this graph are positive, it will be electrically equivalent to one of the four graphs in
Figure~\ref{the-four-it-becomes}, all of which are recoverable.}
\label{example-k4}
\end{figure}
\begin{figure}
\centering
\def \svgwidth{5in}
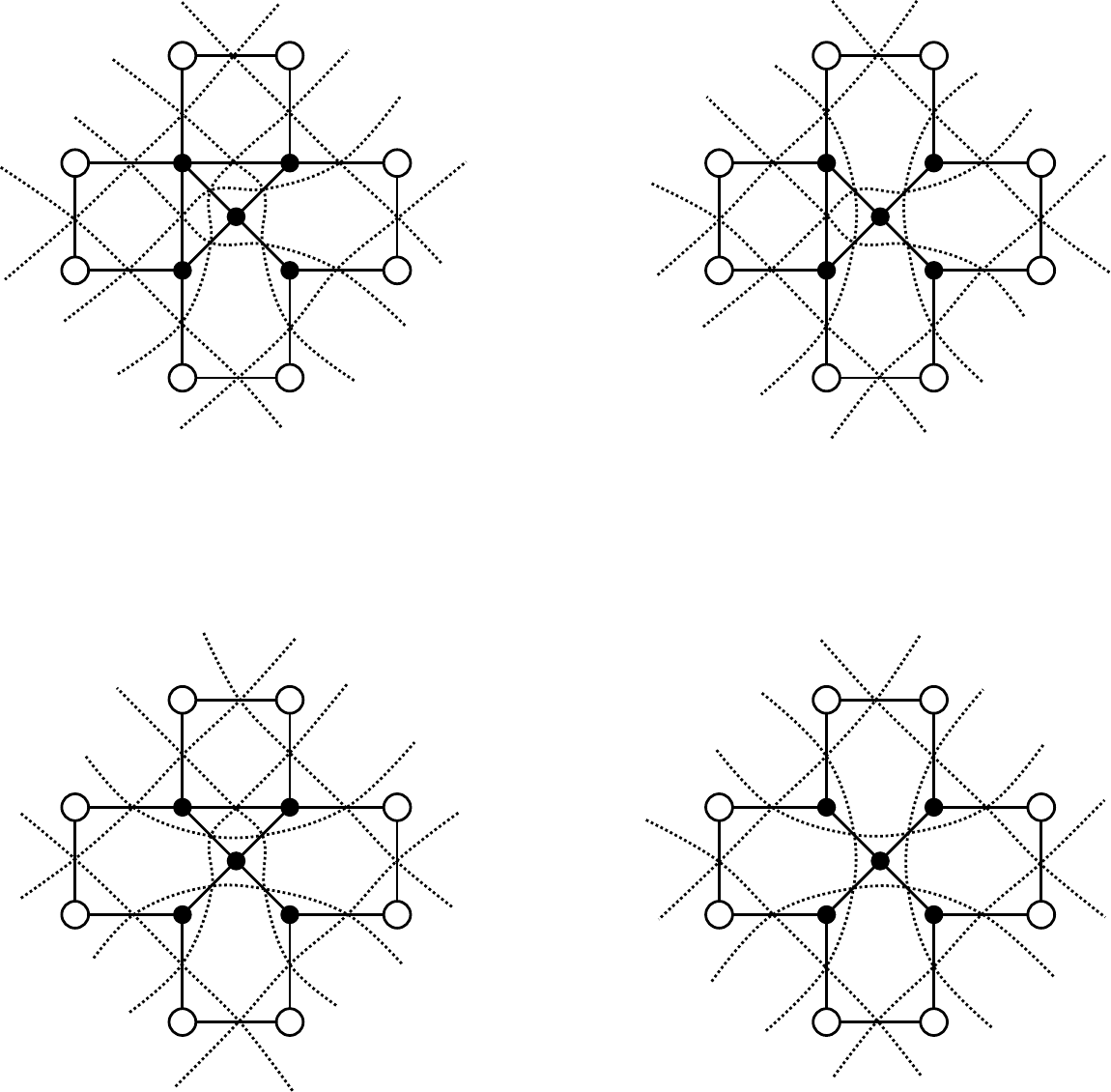
\caption{These four graphs are recoverable.  Furthermore, they are electrically distinguishable from each other, by Lemma~\ref{z-sequence}.}
\label{the-four-it-becomes}
\end{figure}

\begin{definition}
Let $M_1$ and $M_2$ be two critical medial graphs, embedded on the same disks with boundary vertices in the same locations.  Then we say that
$M_1$ and $M_2$ are \emph{equivalent} if for every boundary vertex $v$, the geodesics $g_{M_1}$ of $M_1$ and $g_{M_2}$ of $M_2$ beginning at $v$
end at the same boundary vertex.
\end{definition}

\begin{lemma}\label{z-sequence}
Let $M_1$ and $M_2$ be two colored critical medial graphs with conductances, with the same boundary vertices.  If $M_1$ and $M_2$ have the same
boundary relation, then $M_1$ and $M_2$ are equivalent.
\end{lemma}
\begin{proof}
Let $w$ be one of the common boundary vertices of $M_1$ and $M_2$.  Let $g_i$ be the geodesic in $M_i$ with one endpoint at $w$, and
let $v_i$ be the other endpoint of $g_i$.  Suppose for the sake of contradiction that $v_1 \ne v_2$.  Let $H$ be the half-plane
in $M_1$ on the side of $g_1$ containing $v_2$, and let $b$ be the cell of $M_1$ that is by $w$ but not in $H$, as in Figure~\ref{fig:final-proof}.
By Lemma~\ref{problem-set-up}, we can find sets of boundary
cells $T \subseteq S$, with $b \in S \setminus T$, $S$ and $T$ safe, $\overline{T} = H$, and $\overline{S}$ equal to the whole medial graph $M_1$.
So we can set boundary values of $0$ on $T$, $1$ at $b$, and whatever we like on $S \setminus \{b\}$, and there will be at least one
labelling of $M_1$ with these boundary values.  In fact, since $\overline{T} = H$, the labelling must take the value zero on all of $H$,
so we can even assert boundary values of zero on every boundary cell in $H$.

In particular, we see that for $M_1$, there is at least one set of valid boundary values that takes the value $1$ at $b$, and $0$ at every boundary
cell in $H$.  We claim that this fails in $M_2$, a contradiction.
To see this, let $Q$ be the set of boundary cells in $M_2$ corresponding to the boundary cells of $H$.  That is, $Q$ consists of all
the boundary cells in $M_2$ along the boundary arc from $w$ to $v_1$ that passes through $v_2$.  Let $b'$ be the cell of $M_2$ corresponding
to $b$.  If $M_1$ and $M_2$ had the same boundary relation, then there should be a labelling of $M_2$ which takes the value
$0$ at every cell in $Q$, and $1$ at $b'$.  But this is impossible because $b' \in \overline{Q}$, so that boundary values of zero on $Q$ should
force a value of zero at $b'$.  To see that $b' \in \overline{Q}$, note first that $Q$ is connected, so $\overline{Q}$ is connected and thus convex.
But $b'$ is directly across $g_2$ from a cell of $Q$, so if any geodesic separates $b'$ from $\overline{Q}$, it would be $g_2$, which is impossible
because cells on both sides of $g_2$ are in $Q$, specifically the two cells around $v_2$.  So no geodesic separates $b'$ from $Q$,
and $b' \in \overline{Q}$.

So the assumption that $v_1 \ne v_2$ was false.  Consequently $g_1$ and $g_2$ have the same start and endpoints.
\end{proof}
\begin{figure}
\centering
\subfloat[$M_1$]{\label{final-proof-1}
\def \svgwidth{2in}
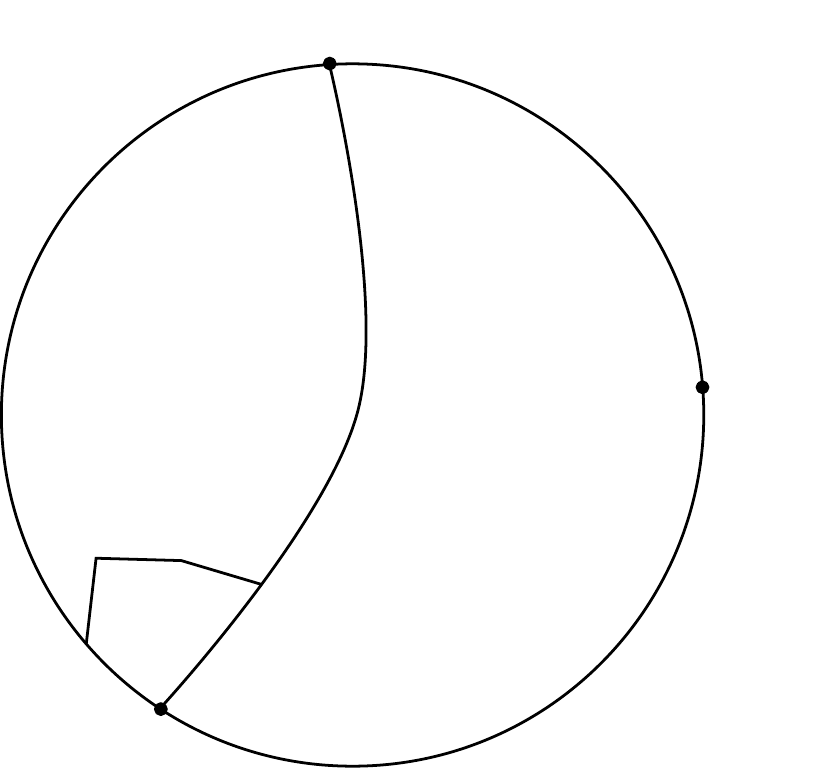}
\qquad
\subfloat[$M_2$]{\label{final-proof-2}
\def \svgwidth{2.25in}
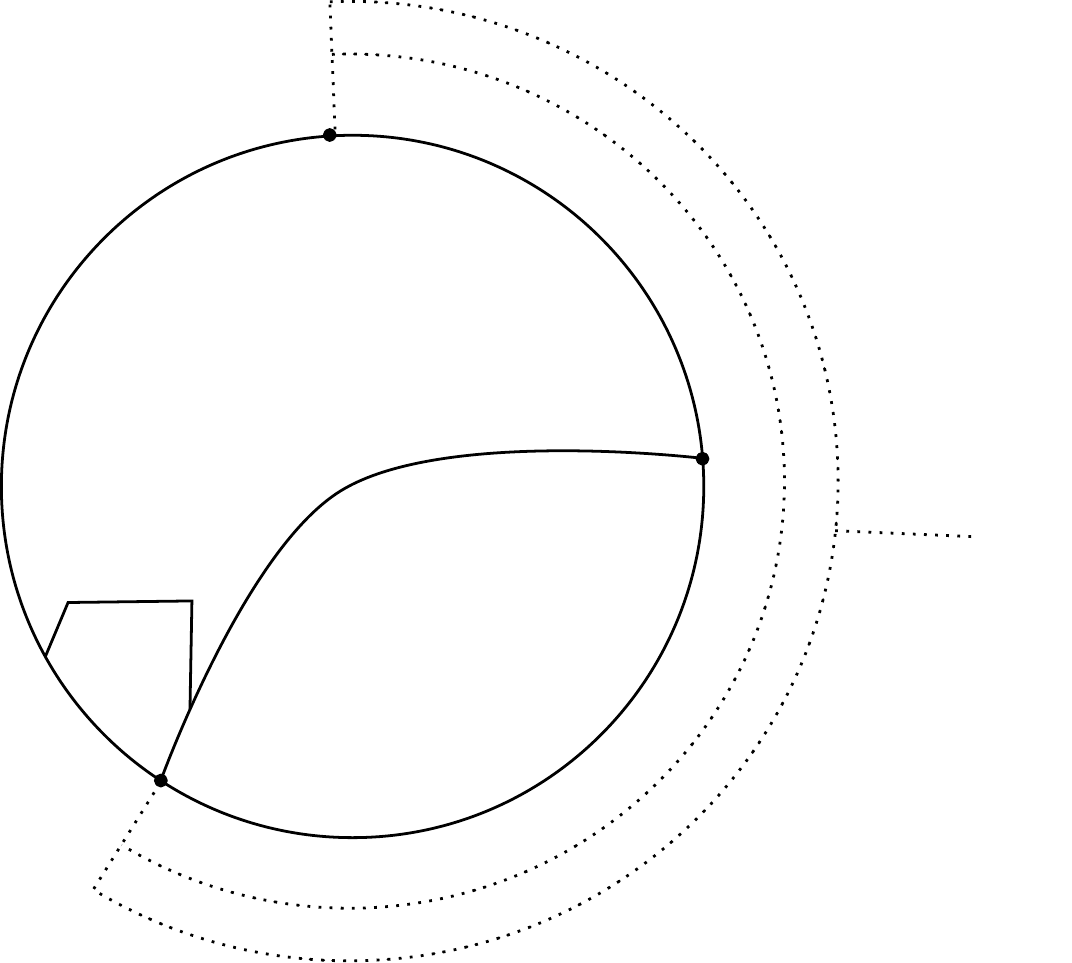}
\caption{Two nonequivalent medial graphs.}
\label{fig:final-proof}
\end{figure}

By Lemma~\ref{z-sequence} and an examination of the medial graphs of Figure~\ref{the-four-it-becomes}, we see that all are electrically distinct, and consequently
the original graph of Figure~\ref{example-k4} is weakly recoverable.

In contrast, the graph of Figure~\ref{generic-nonrecoverable} is not recoverable because in one of the cases, it is equivalent to Figure~\ref{generic-nonrecoverable-1},
which is not recoverable.  To verify this, note that we can turn a ``four-star'' into a $K_4$ as in Figure~\ref{four-star-k},
without producing negative or zero conductances.  This map is invertible (because the ``four-star'' is recoverable),
so if Figure~\ref{generic-nonrecoverable} were recoverable, then Figure~\ref{generic-nonrecoverable-1} would also be recoverable, but it is not since its medial graph
clearly contains a lens.
\begin{figure}
\centering
\def \svgwidth{4.5in}
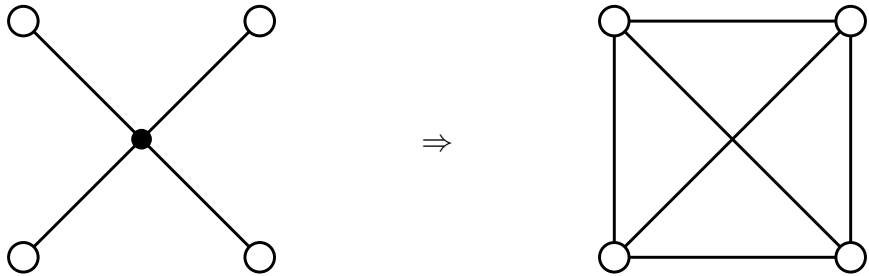
\caption{Turning a ``star'' into a complete graph.}
\label{four-star-k}
\end{figure}

However, we can say that Figure~\ref{generic-nonrecoverable} is \emph{generically recoverable}, because generically it is equivalent
to Figure~\ref{generic-nonrecoverable-2} which is weakly recoverable.  In fact, the map from conductances on Figure~\ref{generic-nonrecoverable} to response matrix
is birational.
\begin{figure}
\centering
\subfloat[~]{\label{generic-nonrecoverable}
\def \svgwidth{1.25in}
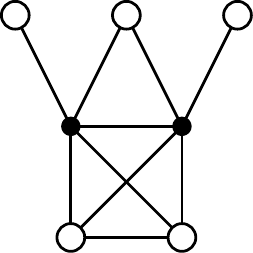}
\subfloat[~]{\label{generic-nonrecoverable-1}
\def \svgwidth{1.5in}
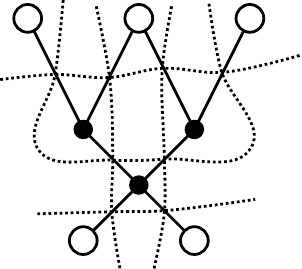}
\subfloat[~]{\label{generic-nonrecoverable-2}
\def \svgwidth{1.5in}
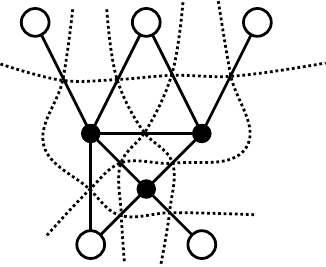}
\caption{Of these three graphs, only \ref{generic-nonrecoverable-2} is recoverable, though \ref{generic-nonrecoverable} is \emph{generically} recoverable.}
\label{fig:generic-recoverability}
\end{figure}

\subsection{Knot Theory}
Another place where negative conductances naturally appear is in knot-theoretic studies of tangles.  As explained in \cite{GoldmanKauffmanElectric},
we can associate a medial graph (with conductances) to a picture of a knot-theoretic tangle, by viewing the strands in the tangle as the geodesics,
and assigning a conductance of $\pm 1$ at each crossing, with sign determined by which of the two strands is on top.  The electrical
behavior of the network is invariant under the Reidemeister moves, making it a topological invariant.  In the case of knots, which
can be thought of as tangles with no loose ends, the resulting electrical network has no boundary vertices, so no useful information is obtained this way.
In the case of tangles with four loose ends,
the resulting electrical network has two boundary vertices, so its electrical behavior is summarized by a single number, the resistance.  For the well-known
family of \emph{rational tangles} (see \cite{GoldmanKauffman}), the resistance is a complete invariant of the tangle.  In particular, up to equivalence there
is a unique rational tangle for every element of $\mathbb{Q} \cup \{\infty\}$.

Combining facts from \cite{GoldmanKauffmanElectric} and \cite{GoldmanKauffman} with our results on signed conductances, one easily obtains the following
novelty theorem:
\begin{theorem}\label{silly-theorem}
Fix some pseudoline arrangement (i.e., critical medial graph) with $n$ pseudolines.
View the geodesics as strands in a knot diagram or tangle diagram.  At each crossing
in the original pseudoline arrangement, substitute a rational tangle not equivalent to $0$ or $\infty$.  Interpret the resulting diagram as a tangle with
$2n$ loose ends.  Then given the initial pseudoline arrangement and the resulting tangle, one can say which rational tangle was substituted at each vertex.
\end{theorem}

As an example, the two tangles in Figure~\ref{tangles-ingredients}(a) are both rational tangles not equivalent to $0$ or $\infty$.
Substituting them into the pseudoline arrangement of Figure~\ref{tangles-ingredients}(b) in two different ways, Theorem~\ref{silly-theorem} implies
that the two tangles in Figure~\ref{tangles-final} are not equivalent.

\begin{figure}
\centering
\def \svgwidth{3in}
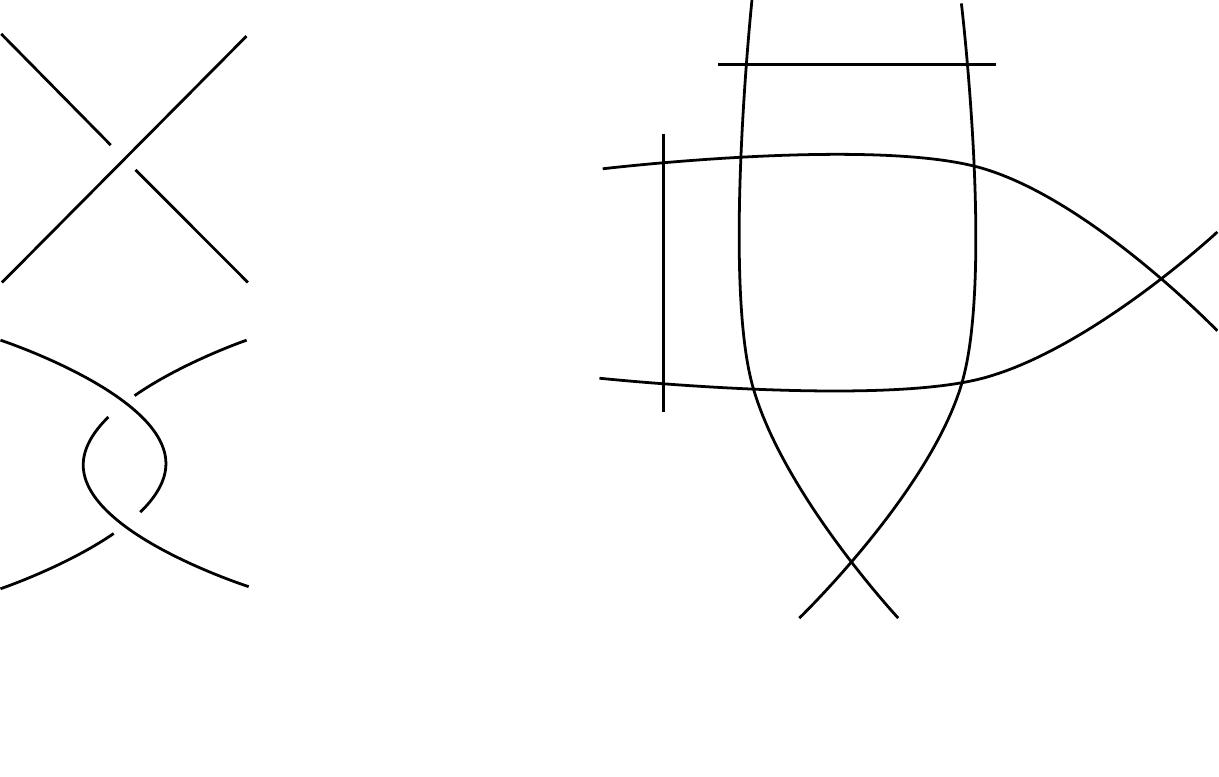
\caption{Two rational tangles, and a pseudoline arrangement.}
\label{tangles-ingredients}
\end{figure}
\begin{figure}
\centering
\def \svgwidth{4in}
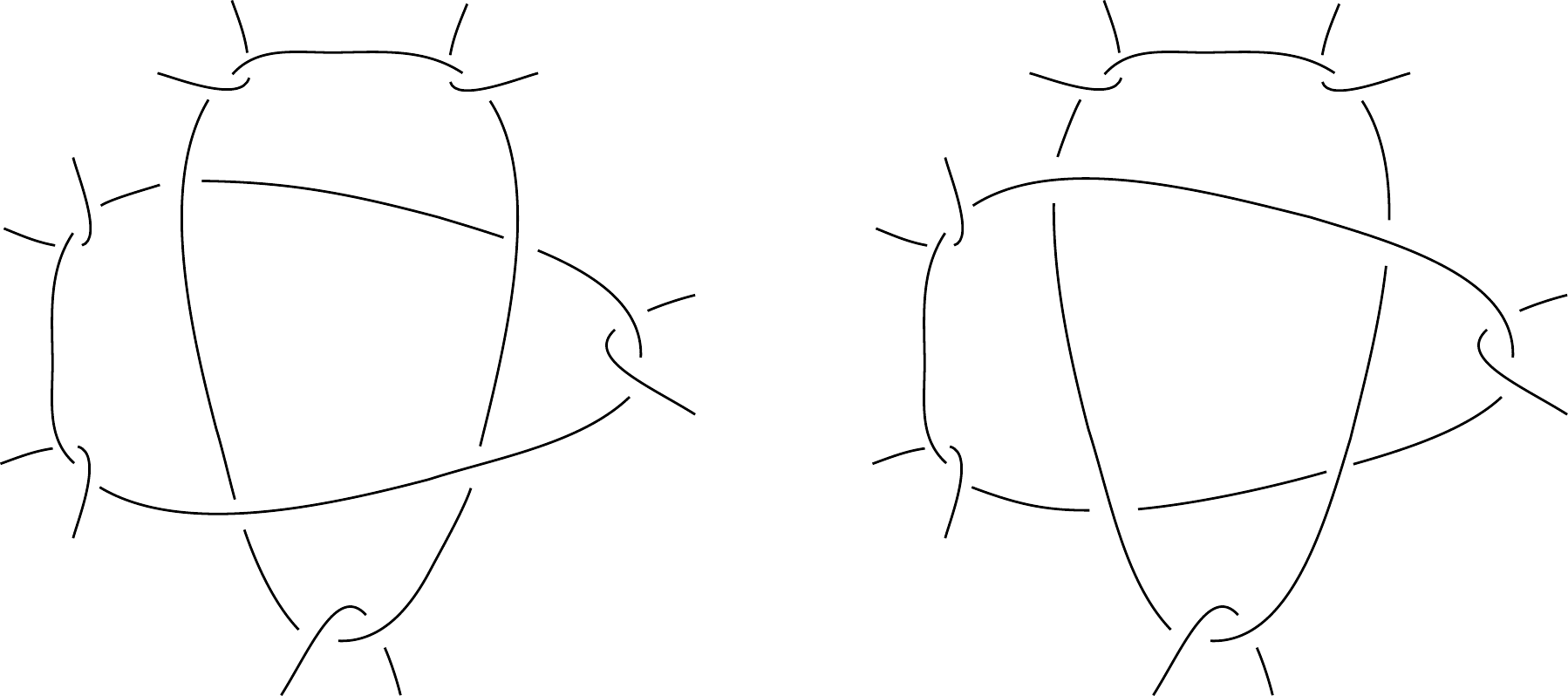
\caption{Some eight-strand tangles obtained by replacing the intersections of Figure \ref{tangles-ingredients}(b) with rational tangles.}
\label{tangles-final}
\end{figure}

It might be interesting to consider whether
Theorem~\ref{silly-theorem} remains true if ``rational tangle'' is replaced by some broader class of 2-strand tangles.

\section{The Electrical Linear Group Perspective}\label{sec:el2n}
In \cite{ElectricalLieTheory}, Lam and Pylyavskyy define a Lie algebra $\mathfrak{el}_{2n}$ whose associated Lie group $EL_{2n}$ is called the \emph{electrical linear group}.  This group has generators $u_i(a)$ for $1 \le i \le 2n$ and $a \in \mathbb{R}$, and these generators happen to satisfy the equations
\begin{equation} u_i(a)u_i(b) = u_i(a+b) \label{el-rel-1}\end{equation}
\begin{equation} u_i(a)u_j(b) = u_j(b)u_i(a)\label{el-rel-2}\end{equation}
when $|i - j| > 1$, and
\begin{equation} u_i(a)u_j(b)u_i(c) = u_j(bc/(a+c+abc)) u_i(a+c+abc) u_j(ac/(a+c+abc)) \label{el-rel-3}\end{equation}
when $|i - j| = 1$ and the denominators make sense.  This group is isomorphic to the symplectic group $Sp_{2n}$.  One isomorphism goes as follows: let $e_{ij}$ denote
the $n \times n$ matrix with a 1 in the $i$th row and $j$th column, and zeros elsewhere.  Then in block notation, we will take
\[ u_{2i-1}(a) = \begin{pmatrix}
I & -ae_{ii} \\ 0 & I
\end{pmatrix} \in Sp_{2n}
\]
for $1 \le i \le n$,
\[ u_{2i}(a) = \begin{pmatrix}
I & 0 \\
-a(e_{ii} + e_{i+1,i+1} - e_{i,i+1} - e_{i+1,i}) & I \\
\end{pmatrix} \in Sp_{2n}\]
for $1 \le i < n$, and
\[ u_{2n}(a) = \begin{pmatrix}
I & 0 \\
-ae_{nn} & I
\end{pmatrix} \in Sp_{2n}.\]
Within the electrical linear group $EL_{2n}$ we have the \emph{electrically nonnegative part} $(EL_{2n})_{\ge 0}$ which is the submonoid generated
by the $u_i(a)$ for $a \ge 0$.  The following fact is essentially Proposition 2.4 in \cite{ElectricalLieTheory}.
\begin{proposition}\label{variant-one}
For $\mathbf{i} = i_1 \ldots i_l$ a word in the alphabet $\{1,\ldots,2n\}$, let $\psi_{\mathbf{i}} : \mathbb{R}^l_{>0} \to EL_{2n}$ be the map
given by
\[ \psi_{\mathbf{i}}(a_1,\ldots,a_l) = u_{i_1}(a_1) \cdots u_{i_n}(a_n) \in EL_{2n}.\]
Then the following are true:
\begin{itemize}
\item $\psi_{\mathbf{i}}$
is injective if and only if $s_{i_1} \cdots s_{i_n}$ is a reduced word in the symmetric group $S_{2n+1}$, where $s_i$ denotes the transposition
of $i$ and $i+1$.
\item When $s_{i_1} \cdots s_{i_n}$ is a reduced word for some $w \in S_{2n+1}$, the image of $\psi_{\mathbf{i}}$
depends only on $w$.
\item Letting $E(w)$ denote this image, $(EL_{2n})_{\ge 0}$ can be written as a disjoint union $(EL_{2n})_{\ge 0} = \coprod_{w \in S_{2n+1}} E(w)$.
\end{itemize}
\end{proposition}
This proposition is of interest for a couple of reasons.  First, it is a disguised restatement of the main results of \cite{ReseauxElectriques} and
\cite{CIM}.\footnote{One thinks of the $u_{2i-1}(a)$ as the operations of adding
a boundary spike of resistance $a$ at the $i$th boundary vertex, and $u_{2i}(a)$ as the operation
of adding a boundary-to-boundary edge of conductance $a$ between the $i$th and $i+1$st boundary vertices.  Relation (\ref{el-rel-3}) corresponds to the
Y-$\Delta$ transformation.
Injectivity of $\psi_{\mathbf{i}}$ becomes
equivalent to the weak recoverability of a certain graph.  The first bullet point of the proposition
is tantamount to the fact that a graph is weakly recoverable if and only if its
medial graph is critical.  The second bullet point follows easily from the relations (\ref{el-rel-1}-\ref{el-rel-3}), together with a well-known theorem
about reduced words in Coxeter groups.  The third bullet point has to do with Lemma~\ref{z-sequence}.}
Second, it is completely analogous to similar facts about factorization schemes for nonnegative upper triangular unipotent matrices, which we now recall.
Let $U$ be the group of upper triangular unipotent matrices in $SL_{2n+1}$, and let
$U_{\ge 0}$ be the submonoid of totally nonnegative ones.  (A matrix is \emph{totally nonnegative}
if every minor is nonnegative.  We refer the reader
to \cite{TotalPosTandP} or \cite{LusztigTotalPosReductiveGroups} for more information on totally nonnegative matrices, and the generalizations of this notion
to other semisimple Lie groups.)  The group $U$ is generated by elementary matrices of the form $x_i(a)
= I + ae_{i,i+1}$, where $a \in \mathbb{R}$ and $1 \le i < n$.  
These generators satisfy relations analogous to (\ref{el-rel-1}-\ref{el-rel-3}), namely
\begin{equation*} x_i(a)x_i(b) = x_i(a+b) \end{equation*}
\begin{equation*} x_i(a)x_j(b) = x_j(b)x_i(a) \end{equation*}
when $|i - j| > 1$, and
\begin{equation*} x_i(a)x_j(b)x_i(c) = x_j(bc/(a+c)) x_i(a+c) x_j(ac/(a+c)) \end{equation*}
when $|i - j| = 1$ and the denominators make sense.
The submonoid $U_{\ge 0}$ turns out to be generated by $x_i(a)$ for $a > 0$, just as $(EL_{2n})_{\ge 0}$ is generated by the $u_i(a)$ for $a > 0$.  Moreover, the following
analogue of Proposition~\ref{variant-one} holds:
\begin{proposition}\label{variant-two}
For $\mathbf{i} = i_1 \ldots i_l$ a word in the alphabet $\{1,\ldots,2n\}$, let $\psi_{\mathbf{i}} : \mathbb{R}^l_{>0} \to U$ be the map
given by
\[ \psi_{\mathbf{i}}(a_1,\ldots,a_l) = x_{i_1}(a_1) \cdots x_{i_n}(a_n) \in EL_{2n}.\]
Then the following are true:
\begin{itemize}
\item $\psi_{\mathbf{i}}$
is injective if and only if $s_{i_1} \cdots s_{i_n}$ is a reduced word in the symmetric group $S_{2n+1}$, where $s_i$ denotes the transposition
of $i$ and $i+1$.
\item When $s_{i_1} \cdots s_{i_n}$ is a reduced word for some $w \in S_{2n+1}$, the image of $\psi_{\mathbf{i}}$
depends only on $w$.
\item Letting $E(w)$ denote this image, $U_{\ge 0}$ can be written as a disjoint union $U_{\ge 0} = \coprod_{w \in S_{2n+1}} E(w)$.
\end{itemize}
\end{proposition}
This is essentially Proposition 2.7 of \cite{LusztigTotalPosReductiveGroups}.

Lusztig's proof of Proposition~\ref{variant-two} also establishes the following strengthening of the first bullet point of Proposition~\ref{variant-two}:
\begin{proposition}\label{variant-three}
If $\mathbf{i} = i_1 \ldots i_l$ is a reduced word, then the map $\mathbb{R}_{\ne 0}^l \to U$
\[ (a_1, \ldots, a_l) \mapsto x_{i_1}(a_1) \cdots x_{i_l}(a_l)\]
is injective.
\end{proposition}
The image of this map may depend on the choice of the reduced word, however.  The analogous statement in the setting of electrical networks
is true, though it doesn't seem to be noted in \cite{ElectricalLieTheory}:
\begin{proposition}\label{variant-four}
If $\mathbf{i} = i_1 \ldots i_l$ is a reduced word, then the map $\mathbb{R}_{\ne 0}^l \to EL_{2n+1}$
\[ (a_1, \ldots, a_l) \mapsto u_{i_1}(a_1) \cdots u_{i_l}(a_l)\]
is an injection.
\end{proposition}
This is more or less a disguised version of Theorem~\ref{main-result} in the case of linear but potentially negative conductances.

More generally, we can disguise Theorem~\ref{main-result} in the following fashion.  Let $\mathbb{S}$ denote the set of all bijections $f : \mathbb{R} \to \mathbb{R}$
such that $f(0) = 0$.  For every $f \in \mathbb{S}$ and $1 \le i \le 2n$, let
$u_i(f)$ be the bijection $\mathbb{R}^{2n} \to \mathbb{R}^{2n}$ defined as follows:
\[ u_{2i-1}(f)(v_1, \ldots, v_n, c_1, \ldots, c_n) = (v_1, \ldots, v_{i-1}, v_i - f(c_i), v_{i+1}, \ldots, v_n, c_1, \ldots, c_n)\]
for $1 \le i \le n$,
\begin{align*} u_{2i}(f)&(v_1, \ldots, v_n, c_1, \ldots, c_n) \\
 &= (v_1, \ldots, v_n, c_1, \ldots, c_{i-2}, c_i - f(v_i - v_{i+1}), c_{i+1} + f(v_i - v_{i+1}), c_{i+2}, \ldots, c_n)\end{align*}
for $1 \le i < n$, and
\[ u_{2n}(f)(v_1, \ldots, v_n, c_1, \ldots, c_n) = (v_1, \ldots, v_n, c_1, \ldots, c_{n-1}, c_n - f(v_n)).\]
Note that if $f$ is multiplication by $a$ for some $a \in \mathbb{R}_{\ne 0}$,
then $u_i(f)$ is a linear function for every $i$, and the corresponding matrix is the same $u_i(a)$ as above.
In the case where $f$ is smooth, $u_i(f)$ will always be a symplectomorphism.  We do not assume that $f$ is continuous, however.

Then Theorem~\ref{main-result} implies the following generalization of Proposition~\ref{variant-four}:
\begin{proposition}\label{variant-five}
If $\mathbf{i} = i_1 \ldots i_l$ is a reduced word, then the map from $\mathbb{S}^l$ to bijections on $\mathbb{R}^{2n}$, sending
\[ (f_1, \ldots, f_l) \mapsto u_{i_1}(f_1) \circ \cdots \circ u_{i_l}(f_l), \]
is an injection.
\end{proposition}

Interestingly, there is also an analog of this for the setting of upper triangular unipotent matrices.  For $1 \le i \le 2n$, let $x_i(f)$ be the
bijection $\mathbb{R}^{2n+1} \to \mathbb{R}^{2n+1}$ defined as
\[ x_i(f)(z_1, \ldots, z_n) = (z_1, \ldots, z_{i-1}, z_i + f(z_{i+1}), z_{i+1}, \ldots, z_n).\]
In the case where $f$ is multiplication by some $a \in \mathbb{R}_{\ne 0}$, the map $x_i(f) : \mathbb{R}^{2n+1} \to \mathbb{R}^{2n+1}$ is linear and the corresponding
matrix is $x_i(a)$ as above.
\begin{proposition}\label{variant-six}
If $\mathbf{i} = i_1 \ldots i_l$ is a reduced word, then the map from $\mathbb{S}^l$ to bijections on $\mathbb{R}^{2n+1}$, sending
\[ (f_1, \ldots, f_l) \mapsto x_{i_1}(f_1) \circ \cdots \circ x_{i_l}(f_l) \]
is an injection.
\end{proposition}
This can be proven in the same way as Theorem~\ref{main-result}, though a bit of ingenuity is needed to make the arguments of
\S \ref{sec:convexity} and \S \ref{sec:nonlinear-recovery} carry over to this context.  There are probably simpler proofs of Propositions \ref{variant-four}
through \ref{variant-six}.

The factorization schemes for totally nonnegative upper triangular unipotent matrices are part of a broader story of factorization schemes for totally nonnegative matrices.
In light of this, we make the following conjecture, which would generalize results of Lusztig~\cite{LusztigTotalPosReductiveGroups}:
\begin{conjecture}
For $1 \le i \le 2n$, and $f \in \mathbb{S}$, let $y_i(f)$ and $h_i(f)$ be the bijections $\mathbb{R}^{2n+1} \to \mathbb{R}^{2n+1}$ given by
\[ y_i(f)(z_1,\ldots,z_n) = (z_1,\ldots,z_i,z_{i+1} + f(z_i), z_{i+2}, \ldots, z_n)\]
\[ h_i(f)(z_1,\ldots,z_n) = (z_1,\ldots,z_{i-1},f(z_i),z_{i+1},\ldots,z_n).\]
If $i_1 \ldots i_l$ is a reduced word and so is $j_1 \ldots j_m$, then the map on $\mathbb{S}^{l + (2n+1) + m}$ sending
\[ (f_1, \ldots, f_l, g_1, \ldots, g_{2n+1}, k_1, \ldots, k_m) \in \mathbb{S}^{l + (2n+1) + m}\]
to
\[ x_{i_1}(f_1) \circ \cdots \circ x_{i_l}(f_l) \circ h_1(g_1) \circ \cdots \circ h_{2n+1}(g_{2n+1})
\circ y_{j_1}(k_1) \circ \cdots \circ y_{j_m}(k_m)\]
should be injective.
\end{conjecture}

\section{Closing Remarks}
We have shown that if $\Gamma$ is a circular planar graph, then $\Gamma$ is strongly recoverable if and only if its medial graph is critical.
This extends results of \cite{ReseauxElectriques} and \cite{CIM}, who showed the same for weak recoverability.  An interesting
corollary is that $\Gamma$ is weakly recoverable if and only if it is strongly recoverable.
An obvious conjecture is that this holds for general $\Gamma$, not necessarily circular planar.  However, it may be necessary to put further restrictions
on conductance functions (like monotonicity or smoothness).  For example, monotonicity and smoothness are almost certainly
sufficient to guarantee non-linear recoverability in the highly non-planar 3-dimensional lattice graphs considered in \cite{ThreeDeeLattice}, by a simple argument
unrelated to the one used in the present paper.

The recovery algorithm presented above makes surprisingly few assumptions on conductance functions.  This suggests that there might
be interesting applications of our method.  Letting currents and voltages live in finite fields might have interesting combinatorial applications.
In the linear case, we have shown that the rational map from conductances to response matrices remains injective and somewhat well-defined
when extended from $(\mathbb{R}_{> 0})^n$ to $(\mathbb{C} \setminus \{0\})^n$.  This is a purely algebraic statement, which might have some uses.

We also note that Ian Zemke has generalized some of the results and approaches of this paper to infinite networks in \cite{Zemke}.

\subsection{Acknowledgments}
This work was done in part during 
REU programs at the University of Washington in 2010 and 2011, and in part while supported by the NSF Graduate Research Fellowship Program in the Autumn of 2011
and Spring of 2012.  The author would like to thank Jim Morrow for introducing him to the eletrical recovery problem, and Richard Kenyon and David Wilson for informing
the author of many of the recent papers on the subject, including \cite{ElectricalLieTheory}.

\bibliographystyle{plain}
\bibliography{Master}{}

\end{document}